\documentclass[english,11pt]{amsart}
\usepackage[T1]{fontenc}
\usepackage[latin9]{inputenc}
\usepackage{babel}
\usepackage{geometry,graphicx} 
\usepackage{amsthm}
\usepackage{amsmath}
\usepackage{amssymb}
 \usepackage{amsfonts}
 \usepackage{amscd}
 \usepackage[all]{xy} 
\usepackage[dvipsnames]{xcolor}
\usepackage{dblfloatfix}
\usepackage{pdfsync}
\usepackage{hyperref}
\usepackage{url}
\usepackage{multirow}
\usepackage[sc]{mathpazo}

\usepackage{floatrow}

\makeatletter
\newtheorem{theorem}{Theorem}[section]
\newtheorem*{theorem*}{Theorem}

\newtheorem{prop}[theorem]{Proposition}

{\theoremstyle{definition} \newtheorem{defn}[theorem]{Definition}}
{\theoremstyle{remark} \newtheorem{remark}[theorem]{Remark}
\newtheorem{example}[theorem]{Example}
\newtheorem{notation}[theorem]{Notation}
\newtheorem{algorithm}[theorem]{Algorithm}
}

\newcommand{\cW}{\mathcal W}
\newcommand{\wR}{R}  
\newcommand{\bfW}{{\mathbf{W}}}

\newcommand{\bfE}{{\mathbf{E}}}
\newcommand{\bfe}{\mathbf{e}}

\newcommand{\bfm}{\mathbf{m}}
\newcommand{\bfd}{\mathbf{d}}
\newcommand{\bfw}{\mathbf{w}}

\newcommand{\bfHbar}{{\mathbf{H}}}

\newcommand{\sM}{\mathcal M}

\newcommand{\mfW}{\mathfrak{W}}

\newcommand{\bC}{\mathbb C}

\newcommand{\cH}{\mathcal H}

\newcommand{\cV}{\mathcal V}

\newcommand{\cM}{\mathcal M}

\newcommand{\cR}{\mathcal R}

\newcommand{\cX}{\mathcal X}
\newcommand{\cY}{\mathcal Y}

\newcommand{\cL}{\mathcal L}
\newcommand{\cK}{\mathcal K}
\newcommand{\cC}{\mathcal C}
\newcommand{\cG}{\mathcal G}

\newcommand{\cS}{\mathcal S}
\newcommand{\cZ}{\mathcal Z}

\newcommand{\sF}{\mathcal F}
\newcommand{\bP}{\mathbb{P}}

\newcommand{\xSlice}{[x]}

\newcommand{\ySlice}{[y]}

\newcommand{\xySlice}{[xy]}

\newcommand{\sH}{\mathcal{G}}
\newcommand{\Sym}{\rm Sym}

\newcommand{\xBar}{{\bf [x]}}
\newcommand{\yBar}{{\bf [y]}}
\newcommand{\zBar}{{\bf [z]}}

\newcommand{\ppp}{\mathbb{P}^{n_1}\times\cdots\times\mathbb{P}^{n_k}}

\newcommand{\sq}{\square}

\DeclareMathOperator{\Rand}{Rand}

\DeclareMathOperator{\trace}{Trace}

\DeclareMathOperator{\system}{System}

\DeclareMathOperator{\points}{Points}

\DeclareMathOperator{\slice}{Slice}

\begin{document}

\title{Multiprojective witness sets and a trace test}


\author[J.~D.~Hauenstein]{Jonathan D.~Hauenstein}
\address{Department of Applied \& Computational Mathematics \& Statistics\\
         University of Notre Dame\\
         Notre Dame, IN  46556\\         
         USA}
\email{hauenstein@nd.edu}
\urladdr{\url{http://www.nd.edu/\~jhauenst}}

\author[J.~I.~Rodriguez]{Jose Israel Rodriguez}
\address{Department of Statistics\\
         University of Chicago\\
         Chicago, IL 60637\\         
         USA}
\email{JoIsRo@uchicago.edu}
\urladdr{\url{http://home.uchicago.edu/\~joisro}}



\begin{abstract}
\noindent
In the field of numerical algebraic geometry, 
positive-dimensional solution sets of systems of polynomial equations 
are described by witness sets.
In this paper, we define multiprojective witness sets
which encode the multidegree information of an 
irreducible multiprojective variety. 
Our main results generalize the regeneration solving procedure,
a trace test, and numerical irreducible decomposition to
the multiprojective case.  Examples are included to demonstrate
this new approach.  

\end{abstract}
\maketitle

\section{Introduction}

Numerical algebraic geometry contains algorithms for computing and studying
solution sets, called {\em varieties}, of systems of polynomial equations.
Depending on the structure of the equations, 
we can view the variety as either affine or projective.
Witness sets are the numerical algebraic geometric description 
of affine and projective varieties.
If a variety is irreducible, its dimension and degree 
can be recovered directly from a witness set.
When the variety is reducible, one can compute 
witness sets for each of the irreducible components 
thereby producing a numerical irreducible decomposition of the variety.

We will consider multiprojective varieties, which are defined by a 
polynomial system consisting of multihomogeneous polynomials.
Multiprojective varieties naturally arise in many applications
including kinematics \cite{AlgebraicKinematics},
likelihood geometry \cite{HRS12,HS14}, 
and identifiability in tensor decomposition \cite{Tensors}.
In fact, multihomogeneous homotopies \cite{MultiHomHomotopy} 
in numerical algebraic geometry
developed from observing bihomogeneous structure
of the inverse kinematics problem for 6R robots \cite{AlgebraicKinematics}.

The paper is structured as follows. 
In Section~\ref{sectionSWS}, we define multiprojective witness sets and their collections. 
Section~\ref{sectionHomotopy} summarizes homotopy continuation.
Section~\ref{sec:Membership} contains our first main contribution: a membership test using multiprojective witness sets. 
Section~\ref{sec:multiregenerate} contains our second 
main contribution: a generalization of the regeneration algorithm to compute multiprojective witness sets
with examples presented in 
Section~\ref{sec:moreExamples}.
Section~\ref{sectionTrace} contains our third
main contribution: a trace test for multiprojective varieties with examples presented in Section~\ref{appTrace}.
In particular, this trace test provided the motivation
to recompute the BKK bound for Alt's problem (see Section~\ref{sec:Alt}) in Theorem~\ref{thm:NinePtBKK}.

\subsection{Multiprojective witness sets}\label{sectionSWS}
A projective variety intersected with a general linear space has an expected number of solutions.
When the dimension of the linear space is complementary to the variety, the expected number of solutions is finite and is
called the {\em degree} of the variety.
A witness point set for a variety is such a 
finite set of points. 

In the multiprojective setting, there are different types of linear spaces that may be taken
which are related to the Chow ring (see \cite[Chap.~8]{MillerSturmfels}).
In particular, a collection of witness sets 
with all possible linear slicing 
can be used to describe a multiprojective variety. 



\begin{defn}\label{defWS}
Let $\cV$ be a $c$-dimensional irreducible multiprojective variety in $\ppp$ and let
$\bfe=\left(e_{1},e_{2},\dots,e_{k}\right)\in\mathbb{N}^k_{\geq 0}$ such that $c=|\bfe|=e_1+\cdots+e_k$.
An \textit{$\bfe^{\rm th}$ witness set} for $\cV$ is a triple: 
$$\mathbf{W}^{\bfe}\left(\cV\right):=
\left\{ V,	L^{\bfe},		\mathbf{w}^{\bfe}\left(\cV\right)\right\}\quad\text{ where }$$
\begin{enumerate}
\item $V$ is a set of polynomials that forms a \textit{witness system} for $\cV$, i.e., $\cV$ is an irreducible component of the solution set $V=0$.  
\item\label{defSlice} $L^{\bfe} = \displaystyle\bigcup_{i=1}^k \left\{\ell_i^{(1)},\dots,\ell_i^{(e_i)}\right\}$ is a set of $|\bfe|$ linear polynomials
with each $\ell_i^{(j)}$ being a general linear polynomial
in the unknowns associated with $\mathbb{P}^{n_{i}}$
of $\ppp$. The solution set $L^\bfe = 0$ defines a codimension $|\bfe|$ linear space denoted by $\cL^\bfe$.
\item $\mathbf{w}^{\bfe}\left(\cV\right) = \cV\cap\cL^\bfe$
is the \textit{witness point set} for $\cV$ with respect
to $L^\bfe$.
\end{enumerate}
The number of points in $\mathbf{w}^{\bfe}\left(\cV\right)$,
namely $|\mathbf{w}^{\bfe}\left(\cV\right)|$,
is the \textit{degree of $\cV$ with respect to $\bfe$}, 
which we will formally denote by
$\deg\mathbf{W}^\bfe\left(\cV\right)=|\bfw^\bfe(\cV)| \omega^\bfe$.
\end{defn}

\begin{defn}\label{defMPC}
A (complete) \textit{witness set collection} for 
a $c$-dimensional irreducible multiprojective 
variety $\cV\subset\ppp$ is a
formal union of witness sets:
\[
\mathfrak{W}\left(\cV\right)=
\bigsqcup_{
\substack{
\bfe\in\mathbb{N}_{\geq0}^{k} \\
c=|\bfe|}}\mathbf{W}^{\bfe}\left(\cV\right)
\]
with {\em degree} defined to be a formal sum:
\[
\deg\mathfrak{W(\cV)}=
\sum_{\substack{\bfe\in\mathbb{N}_{\geq0}^{k} \\ c=|\bfe|}}
|{\bfw^{{\bfe}}\left(\cV\right)}|\omega^{{\bfe}}.
\]
When the context is clear, 
we will write
$\mathfrak{W},\mathbf{W}^{\bfe},\mathbf{w}^{\bfe}$
for $\mathfrak{W}\left(\cV\right),\mathbf{W}^{\bfe}\left(\cV\right),\mathbf{w}^{\bfe}\left(\cV\right)$,
respectively. 
\end{defn}

With this setup, the {\em (multi)degree} of $\cV$ is
$\deg\mathfrak{W}(\cV)$.

\begin{remark}
One may disregard the terms 
where $|{\bfw}^{\bfe}\left(\cV\right)|=0$
in both the formal union of witness sets and 
the formal sum of degrees.
\end{remark}

\begin{example}\label{parabola}
As an illustrative example, let $\cV$ be the irreducible biprojective
curve in $\bP^1\times\bP^1$ with coordinates $([x_0,x_1],[y_0,y_1])\in\bP^1\times\bP^1$ defined by
$V=\{x_1^2y_0-x_0^2y_1\}$.
The degree of $\cV$ and $\mfW$ is 
$1\omega^{(1,0)}+2\omega^{(0,1)}$,
with the geometric meaning of this observed
in Figure~\ref{fig:parabolaSimple}.
In particular,
since $\cV$ is an irreducible hypersurface,
 $\deg\cV$ can be observed directly from $V$ 
based on the degree in each set of variables, 
i.e.,  $\deg\cV=\deg_y(V)\omega^{(1,0)} + \deg_x(V)\omega^{(0,1)}$.

\begin{figure}
\floatbox[{\capbeside\thisfloatsetup{capbesideposition={left,top},capbesidewidth=8cm}}]{figure}[\FBwidth]
{\caption{\small{
In the affine chart $x_0 = y_0 = 1$,
the parabola intersects the vertical line 
($\cL^{(1,0)}$) at one point 
and the horizontal line ($\cL^{(0,1)}$) at two points.
}
}\label{fig:parabolaSimple}}
{\includegraphics[scale=0.3]{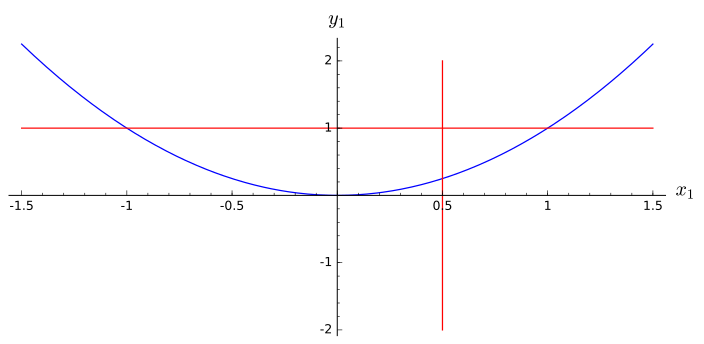}}
\end{figure}
\end{example}

\begin{notation}
Our convention arises from a geometric interpretation based on slicing.
This is the ``reciprocal'' of the algebraic convention,
which is followed by the {\tt multidegree} function in {\tt Macaulay2} \cite{M2}.
For brevity, the degree of a hypersurface $\mathcal{G}$ 
computed by {\tt Macaulay2} is  
$$
\deg{\cG}=\left(|{\bf w}^{\left(n_{1}-1,n_{2},\dots,n_{k}\right)}|,|{\bf w}^{\left(n_{1},n_{2}-1,\dots,n_{k}\right)}|,\dots,|{\bf w}^{\left(n_{1},n_{2},\dots,n_{k}-1\right)}|\right).
$$
Thus, the degree of the hypersurface $\cV$ 
in Ex.~\ref{parabola} may be written as $(2,1)$.  
\end{notation}

\begin{remark}\label{Deflation}
According to Definition \ref{defMPC}, $\bfw^\bfe(\cV)$
is a \textit{set} of points. 
This set contains no information about multiplicity that is 
necessary for describing generically nonreduced components,
i.e., components with multiplicity greater than one.
One can attach the local multiplicity structure in the form
of a Macaulay dual basis, e.g., \cite{DZ05,HSZ13,HMS15}, to the points.
Additionally, one can employ deflation methods, e.g., \cite{Isosingular,LVZ06},
to perform computations on generically~nonreduced~components.
\end{remark}

Just as for classical witness sets, 
e.g., see \cite[Chap.~13]{SW05}, 
we extend the above definitions to 
reducible varieties by taking
formal unions over the irreducible components.

\section{Using Homotopy Continuation for Witness Sets}\label{sectionHomotopy}

A witness set provides information needed 
to perform geometric computations on varieties. 
The tool that permits such computations is
homotopy continuation, which we briefly summarize
in this section.  

\subsection{Homotopies}\label{Sec:Homotopies}

Homotopy continuation is a fundamental tool in numerical algebraic geometry discussed in detail in \cite{BHSW13,SW05}
and implemented in several software packages,
e.g., \cite{Bertini,Hom4PS2,NAG4M2,PHCpack}.
In this manuscript, we employ straight-line homotopies,
e.g., \cite[\S~51]{Munkres}.
One typical use of homotopy continuation below 
is to deform the linear space $\cL^\bfe$ 
as in the witness set~$\bfW^\bfe(\cV)$ 
to another linear space of the same type, say
$\cM^\bfe$, along $\cV$.  
Let $V$ be a witness system for~$\cV$
and suppose that $\cL^\bfe$ and~$\cM^\bfe$ are
defined by linear equations $L^\bfe$ and $M^\bfe$, respectively.
We denote this homotopy by
$$\bfHbar(\cV,\cL^\bfe\to \cM^\bfe):=\begin{cases}
V\\
tL^\bfe+\left(1-t\right)M^\bfe.\end{cases}$$
The set of start points, at $t = 1$, for this homotopy is
the witness point set $\bfw^\bfe(\cV)$.

For membership testing (Section~\ref{sec:Membership}),
one deforms to a general linear space
of type $\bfe$ passing through a given point $\alpha^*$.
The following specifies notation for such a linear 
space.

\begin{notation}\label{notationMembership}
Given a point $\alpha^*\in\ppp$ and $\bfe$,
let $\cL^\bfe_{\alpha^*}$ be a general linear space 
of type $\bfe$ that passes through $\alpha^*$.
\end{notation}

\begin{example}
Let $\cV\subset\bP^1\times\bP^1$ and $V$ as in Ex.~\ref{parabola}.  Let $\bfe = (0,1)$
and $\cL^\bfe$ be the general linear
space defined by $L^\bfe = \{y_0 - 5y_1\}$.
Then, $\bfw^\bfe(\cV)$ consists of two points, namely
$$\bfw^\bfe(\cV) = 
\left\{\left(\left[\sqrt{5}:1\right],[5:1]\right),\left(\left[-\sqrt{5}:1\right],[5:1]\right)\right\}.$$
Consider $\alpha^* = \left([1:1],[1:2]\right)$
and let $\cM^\bfe = \cL^\bfe_{\alpha^*}$
which is defined by 
$M^\bfe = L^\bfe_{\alpha^*} = \{2 y_0 - y_1\}$.  The homotopy
$\bfHbar(\cV,\cL^\bfe\to \cM^\bfe)
= \bfHbar(\cV,\cL^\bfe\to \cL^\bfe_{\alpha^*})$ is thus
$$\left[\begin{array}{c}
x_1^2 y_0 - x_0^2 y_1 \\
t(y_0 - 5 y_1) + (1-t)(2y_0 - y_1) 
\end{array}\right] = 0.$$
Starting, at $t = 1$, with 
the two points $\bfw^\bfe(\cV)$,
the set of endpoints for this homotopy is 
$$\left\{\left(\left[1:\sqrt{2}\right],[1:2]\right),\left(\left[1:-\sqrt{2}\right],[1:2]\right)\right\}.$$
\end{example}

For properly constructed homotopies, called {\em complete homotopies} in \cite{Regen}, each solution path 
defined by the homotopy, say $P_i(t)$, 
is smooth on~$(0,1]$ and thus can be tracked using numerical {\em path tracking} methods, e.g., a predictor-corrector based
approach.  
{\em Endgames}, e.g., see \cite[Ch.~10]{SW05} and \cite{PolyhedralEG}, 
are used to accurately compute the endpoints of the path, i.e., compute $P_i(0)$.

Computationally, we perform path tracking in projective and multiprojective spaces 
by restricting to affine charts.  
That is, one imposes corresponding affine conditions on the coordinates thereby 
fixing a representation of the (multi)projective points. 
For example, in $\bP^1\times\bP^1$, we can perform
computations in $\bC^2\times\bC^2$ by taking affine charts
of the form
$$\square x_0 + \square x_1 + \square = 0\hbox{~~~~and~~~~}\square y_0 + \square y_1 + \square = 0.$$
Here, we follow the convention of
\cite{BHSW13} where $\square$ represents a random
or unspecified complex number.  

\subsection{Randomization}\label{sec:Randomization}

In a witness set for an irreducible variety 
$\cV$, the only condition on the 
polynomial system $V$ is that 
it is a witness system (Def.~\ref{defWS}). 
As in Remark~\ref{Deflation},
deflation techniques can be used to produce
a system of polynomial equations $V$ such 
that $\cV$ is a generically reduced irreducible
component of $V$.  That is, the 
dimension of the null space
of the Jacobian
matrix of $V$ evaluated at a general point of $\cV$ 
is equal to the dimension of $\cV$.
In particular, the number of polynomials in $V$
is greater than or equal to the codimension of $\cV$.
When the number of polynomials in $V$ is
equal to the codimension of~$\cV$, 
homotopies $\bfHbar(\cV,\cL^\bfe\to \cM^\bfe)$
as constructed above are well-constrained systems,
i.e.,~{\em square}.  

If the number of polynomials in $V$ is strictly
greater than the codimension of $\cV$, 
the homtopies are over-determined.  
As discussed in \cite[\S~9.2]{BHSW13}, 
we will employ randomization for over-determined
systems to improve numerical stability.
That is, one performs numerical computations
by replacing $V$ with a generic randomization of $V$,
say $\Rand(V)$.  Hence, $\cV$ is a generically
reduced irreducible component of $\Rand(V)$
in which the number of polynomials in $\Rand(V)$ is
equal to the codimension of $\cV$.  

In the multiprojective setting, we will maintain
multihomogeneity by randomizing with respect 
to specified affine charts by fixing hyperplanes
at infinity.  To that end,
in each $\bP^{n_i}$, suppose that $\cH_{i}$ is 
a general hyperplane defined by the linear polynomial $H_{i}$.
We construct the affine charts by having $H_{i} = 1$
and maintain multihomogeneity in the randomization via $H_{i}$. 
This is demonstrated in the~following.

\begin{example}\label{ex:Randomization}
Consider the following on 
$\bP^2\times\bP^2$ with variables $[x_0,x_1,x_2]$
and $[y_0,y_1,y_2]$, respectively:
$$V = \left[\begin{array}{c}
x_1^3 y_1 - x_2^3 y_2 \\
x_1x_2y_0^2 - x_0^2y_1y_2 \\
x_1^2y_0 - x_0x_2y_2 \\
x_2^2y_0 - x_0x_1y_1
\end{array}\right].
$$
Let $\cV$ be the solution set of $V = 0$,
which is irreducible and has codimension $2$.  
Hence, we aim to construct $\Rand(V)$ consisting
of $2$ polynomials by using the following:
$$H_1= x_0 + 2x_1 - 3x_2 \hbox{~~~and~~~}
H_2 = 2 y_0 - 5 y_1 + 3 y_2.$$
For example, we can take $\Rand(V)$ to have the form
$$\Rand(V) = 
\left[\begin{array}{c}
x_1^3 y_1 - x_2^3 y_2 + \square H_1 (x_1^2y_0 - x_0x_2y_2) + \square H_1(x_2^2y_0 - x_0x_1y_1) \\
x_1x_2y_0^2 - x_0^2y_1y_2 + \square H_2 (x_1^2y_0 - x_0x_2y_2) + \square H_2(x_2^2y_0 - x_0x_1y_1) \\
\end{array}\right]$$
so that both polynomials in $\Rand(V)$ are homogeneous
of degree $(3,1)$ and $(2,2)$, respectively.
\end{example}

One downside of utilizing randomization is the destruction
of sparsity structure.  For example, in Ex.~\ref{ex:Randomization}, $V$ consists of binomials
while $\Rand(V)$ does not.  
Another downside of randomization is the increase of degrees.
In $\bC^N$ or $\bP^N$, one can order the polynomials
based on degrees to minimize the degrees of the polynomials in the randomization, i.e., adding random linear combinations of 
smaller degree polynomials to polynomials of higher degree.
However, in the multiprojective case, the degrees
of the polynomials, each of which is
a vector of integers, need not have a well-ordering.

\section{Membership Test}\label{sec:Membership}

One application of witness sets is to use
homotopy continuation  to decide membership in the 
corresponding variety~\cite{SVW2001}
which was extended to images of algebraic sets
using pseudowitness sets in~\cite{HS13}.  
In this section, we describe our first
main contribution which is 
using a multiprojective witness set collection to test membership in the
corresponding multiprojective variety.

Suppose that $\cV\subset\ppp$ is an irreducible multiprojective variety and we want decide if a
given point $\alpha^*$ is a member of $\cV$.  
The following highlights the difficulty
of producing a membership test for multiprojective varieties,
namely the loss of transversality of slices 
passing through~$\alpha^*$.

\begin{example}\label{ex:NotTransverse}
Let $\cV\subset\bP^2\times\bP^2$ be the irreducible surface defined by 
$$V=\left[\begin{array}{c} x_1y_1 - x_2y_2 \\
x_0y_1y_2 - x_1y_0^2 \\ x_0 y_1^2 - x_2 y_0^2\end{array}\right]$$ 
and $\alpha^* =([1:0:0],[1:0:3])$.
Since $\cV$ is the only irreducible component defined by
$V$, we verify that~$\alpha^*\in\cV$ by observing $V(\alpha^*) = 0$.  

The multiprojective witness set collection for $\cV$ 
consists of three witness sets
$$\mfW(\cV)=\bfW^{(2,0)}\sqcup \bfW^{(1,1)}\sqcup \bfW^{(0,2)}
\hbox{~~with~~} 
\deg \mfW(\cV) = 2\omega^{(2,0)} + 2\omega^{(1,1)} + 1\omega^{(0,2)}.
$$
In particular, for $\bfe=(2,0)$,
the witness set $\bfW^{\bfe}$ has two witness points.
For simplicity, let $\cL^\bfe$ be the 
 general codimension 2 linear space defined 
by $L^\bfe =\{-x_0+3x_1-2x_2,x_0+x_1+3x_2\}$
and~$\cL^\bfe_{\alpha^*}$ be
defined by $L^\bfe_{\alpha^*} = \{x_1,x_2\}$.
The homotopy $\bfHbar(\cV,\cL^\bfe\to \cL^\bfe_{\alpha^*})$ defines two paths, both of which end
at $([1:0:0],[1:0:0])$.  Since $\alpha^*$ is 
not an endpoint of this homotopy,
a natural conclusion based on previous membership
tests \cite{SVW2001,HS13} 
is that $\alpha^*\notin\cV$.  
However, this perceived failure is obtained
since these membership tests are based on
the intersection of the variety and the linear space
passing through the test point $\alpha^*$ to be transverse
at $\alpha^*$ if $\alpha^*$ is indeed
contained in the variety.  
Here, $\cV\cap\cL_{\alpha^*}^\bfe$ is actually 
a positive-dimensional set 
that contains $\alpha^*$, 
i.e., the intersection of the dimension $2$ variety $\cV$ and the codimension $2$ linear space $\cL_{\alpha^*}^\bfe$ is not transverse at $\alpha^*$.
\end{example}

The following algorithm takes into account transversality for a membership test.

\begin{algorithm}\label{alg:MembershipMP}[Membership test using multiprojective witness sets] 
Given a multiprojective witness set collection $\mfW(\cV)$ for an irreducible variety $\cV\subset\ppp$,
determine if a given point $\alpha^*$ is contained in $\cV$.
\begin{enumerate}
 \item For each $\bfW^\bfe(\cV)$ in $\mfW(\cV)$ with $|w^\bfe(\cV)| > 0$
\begin{enumerate}
\item Construct a general linear space 
$\cL_{\alpha^*}^\bfe$ of type $\bfe$ 
that passes through $\alpha^*$.
\item With start points $\bfw^\bfe(\cV)$,
compute the endpoints $\bfE^\bfe$ defined
by $\bfHbar(\cV,\cL^\bfe\to \cL^\bfe_{\alpha^*})$.
\item\label{isMember} If $\alpha^* \in \bfE^\bfe$, return ``$\alpha^*$ is a member of $\cV$.''
\item\label{isolatedMember} If every point in $\bfE^\bfe$ is isolated in 
$\cV\cap\cL^\bfe_{\alpha^*}$ and $\alpha^*\not\in \bfE^\bfe$, return ``$\alpha^*$ is not a member of $\cV$.''
\end{enumerate} 
\item\label{notMember} Return ``$\alpha^*$ is not a member of $\cV$.''
\end{enumerate}
\end{algorithm}

Before proving correctness of Algorithm~\ref{alg:MembershipMP}, we first show that, for each $\alpha^*\in\cV$, there exists
an element $\bfW^\bfe(\cV)\in\mfW(\cV)$ 
with $|\bfw^\bfe(\cV)|>0$ such that
a general slice $\cL_{\alpha^*}^\bfe$ 
is transverse to $\cV$ at $\alpha^*$.

\begin{prop}[Existence of transversal slices]\label{Prop:Transverse}
If $\cV\subset\ppp$ is an irreducible variety 
and~$\alpha^*\in\cV$, there exists $\bfW^\bfe(\cV)\in\mfW(\cV)$
such that $|\bfw^\bfe(\cV)|> 0$
in which a general linear space $\cL_{\alpha^*}^\bfe$
of type $\bfe$ 
passing through~$\alpha^*$
is transverse to~$\cV$ at $\alpha^*$,
i.e., $\alpha^*$ is an isolated point in $\cV\cap\cL_{\alpha^*}^\bfe$.
\end{prop}
\begin{proof}
For simplicity in our constructive proof, 
we write $\alpha^* = (\alpha_1^*,\dots,\alpha_k^*)\in\ppp$ and set 
$$\sM_i: = \bP^{n_1}\times\cdots\times\bP^{n_{i-1}}\times\{\alpha_i^*\}\times\bP^{n_{i+1}}\times\cdots\times\bP^{n_k}.$$
Define $e_1 := \dim \cV - \dim (\cV\cap \sM_1)$
and let $\cK_{1}\subset\ppp$ be a general linear space
passing through $\alpha^*$ that imposes $e_1$ conditions on $\bP^{n_1}$, i.e., $\cK_1$ is defined by $e_1$ general linear equations in the unknowns of $\bP^{n_1}$ that vanish on $\alpha_1^*$.
%
We construct $e_2,\dots,e_k$ recursively.
For $i = 2,\dots,k$, define
$$e_i : = \dim\left(\cV\cap\cM_1\cap\cdots\cap\cM_{i-1}\right) - \dim\left(\cV\cap\cM_{1}\cap\cdots\cap\cM_{i-1}\cap\sM_i\right),$$
and let $\cK_{i}$ be a general linear space
passing through $\alpha^*$ that imposes $e_i$ conditions~on~$\bP^{n_i}$.  

By construction, the intersection  $\cap_{i=1}^k\cK_i$ 
is a general linear space of type $\bfe$ 
that passes through $\alpha^*$. 
Thus, we can take $\cL_{\alpha^*}^\bfe $ to be $\cap_{i=1}^k\cK_i$.
It remains to show that for the constructed $\bfe$, the linear space $\cap_{i=1}^k\cK_i$  
is transverse to~$\cV$ at $\alpha^*$.
Since $e_i$ is the dimension
of the fiber over $\alpha^*_i$ with respect to
$\cV\cap \cK_{1}\cap\cdots\cap\cK_{i-1}$,
which is nonempty and has dimension at most~$n_i$, 
we know $0\leq e_i\leq n_i$. 
If $c = \dim \cV$, then it follows that $c = |\bfe|$
because $\cV$ is irreducible and $\alpha^*\in\cV$.
Since  $\cap_{i=1}^k\cK_i$ is 
sequentially imposing $e_i$ conditions
on the corresponding fibers which have dimension $e_i$,
we know that~$\alpha^*$ is an isolated point 
in~$\cV\cap \cK_1\cap\cdots\cap\cK_k$. 
Upper semicontinuity, e.g., \cite[Thm~A.4.5]{SW05},
provides~\mbox{$|\bfw^\bfe(\cV)|\geq 1$}. 
\end{proof}

The following illustrates this construction.

\begin{example}\label{ex:LoopMembership}
With the setup from Ex.~\ref{ex:NotTransverse},
we follow the proof of Prop.~\ref{Prop:Transverse}.
First,
$$e_1 = \dim \cV - \dim\left(\cV\cap(\{[1:0:0]\}\times\bP^2)\right) = 2 - 1 = 1.$$  
Hence, only one condition on the first $\bP^2$ is needed
to be imposed, say by $\cK_{1}$ 
as in Prop.~\ref{Prop:Transverse}. Next, 
$$e_2 = \dim(\cV\cap\cK_{1}) - \dim\left(
\cV\cap\cK_{1}\cap(\bP^2\times\{[1:0:3]\})\right)
= 1 - 0 = 1$$ 
showing that we also need to impose one condition on
the second $\bP^2$.   Hence, $\bfe = (1,1)$ will yield 
slices that are transversal to $\cV$ at $\alpha^*$.

To illustrate, we consider the linear spaces 
$\cL^\bfe$ and $\cL^\bfe_{\alpha^*}$ 
defined by 
$$L^\bfe = \{x_0 + x_1 + 3x_2, y_0 - 2\sqrt{-1} y_1 - y_2\}
\hbox{~~~and~~~}
L_{\alpha^*}^\bfe = \{x_1 + 3x_2, y_0 - 2\sqrt{-1}y_1 - y_2/3\},$$ respectively.  
The endpoints of the two solution paths
defined by the homotopy
$\bfHbar(\cV,\cL^\bfe\to \cL^\bfe_{\alpha^*})$ 
are~$\alpha^*$ and 
$([3+4\sqrt{-1}:3:-1],[1-2\sqrt{-1}:-1:3])$
with $\alpha^*$ indeed being
an isolated point of $\cV\cap\cL_{\alpha^*}^\bfe$.
\end{example}

We note that the constructive proof of Prop.~\ref{Prop:Transverse} implicitly used an ordering
of the spaces in the multiprojective space $\ppp$.
If the spaces were ordered differently, one could yield
a different~$\bfe$.  

\begin{example}\label{ex:LoopMembership2}
With the setup from Ex.~\ref{ex:NotTransverse},
we follow the proof of Prop.~\ref{Prop:Transverse}
but take the $\bP^2$'s is reverse order, i.e., 
constructing $\bfe = (e_1,e_2)$ by first computing $e_2$
and then computing $e_1$.  In this case, 
$$e_2 = \dim V - \dim\left(\cV\cap(\bP^2\times\{[1:0:3]\})\right) = 2 - 0 = 2$$
and thus $e_1 = 0$.  Hence, $\bfe = (0,2)$
will also yield transversality at 
$\alpha^*$ and therefore can also be used in a membership test.  
In fact, since $\bfw^\bfe(\cV)$ only consists
of one point in this case, only one path is needed
to be tracked to determine that $\alpha^*\in\cV$.
\end{example}

We use Prop.~\ref{Prop:Transverse} to show
that Algorithm~\ref{alg:MembershipMP} provides
a membership test for multiprojective varieties.

\begin{theorem}[Correctness of Algorithm~\ref{alg:MembershipMP}]\label{Thm:Membership}
Algorithm~\ref{alg:MembershipMP} is a valid 
membership test based on multiprojective witness sets.
\end{theorem}
\begin{proof}
If $\bfe$ is chosen such that 
$\alpha^*\in \bfE^\bfe$, 
then we know $\alpha^*\in\cV$ because $\bfE^\bfe \subset \cV\cap\cL^\bfe_{\alpha^*}$
which  justifies Item~\ref{isMember} of Algorithm \ref{alg:MembershipMP}. 
If every point in~$\bfE^\bfe$
is isolated in $\cV\cap\cL_{\alpha^*}^\bfe$, then
coefficient-parameter homotopy theory \cite{CoeffParam}
provides that $\bfE^\bfe = \cV\cap\cL_{\alpha^*}^\bfe$.
Hence, in this case, $\alpha^*\in\cV$ if and only if $\alpha^*\in\bfE^\bfe$ which justifies Item \ref{isolatedMember}. 
By Prop.~\ref{Prop:Transverse}, there exists $\bfe$ such that 
$|\bfw^\bfe(\cV)| \geq 1$ and we have the property that  
$\alpha^*\in\cV$ if and only if $\alpha^*\in\bfE^\bfe$
which justifies Item~\ref{notMember}. 



\end{proof}

Since Algorithm~\ref{alg:MembershipMP} requires
the input of multihomogeneous witness set 
collection $\mfW(\cV)$, the rest 
of the paper is devoted to computing such collections.

\section{Multiregeneration}\label{sec:multiregenerate}
In this section, we will compute a witness set collection of a variety $\cV$. 
Thus, the aim is to compute isolated points of
$\cV\cap\cL^\bfe$ for all possible 
$\bfe$. 
When  $\cV$ is equidimensional, by all possible $\bfe$  we mean $\bfe\in\mathbb{N}_{\geq 0}^k$ such that $|\bfe|=\dim\cV$; otherwise, we mean  $\mathbf{0}\leq\bfe\leq(n_1,\dots,n_k)$. 
We take an equation-by-equation approach called regeneration \cite{Regen,HSW11}. 
The key idea is the following: given a witness set collection for a variety $\cY$ and a hypersurface $\cG$, compute a witness set collection for $\cY\cap\cG$. Iterating this process we will compute $\mfW(\cV)$ with $\cV$ defined as the intersection of hypersurfaces. 
From $\mfW(\cV$), Section~\ref{sec:Decomposition}
describes computing witness sets
for the irreducible components of $\cV$.


By utilizing an approach based on regeneration,
one is actually constructing a sequence of witness 
point sets for subproblems as part of the computation.
When the subproblems have physical meaning, 
the intermediate steps of regeneration provide useful information.
For example, Alt's problem~\cite{Alt,AltBurmester}
counts the number of four-bar linkages 
whose coupler curve passes through nine general
points in the plane (see Section~\ref{sec:Alt}).  
By using a regeneration-based approach, one actually 
solves the two-point, three-point, $\dots$, 
and eight-point problems in the process of
solving Alt's~nine-point~problem.

The key step of multiregeneration is presented in Section~\ref{Sec:RegenKey}. 
We summarize the 
complete
algorithm in Section~\ref{Sec:RegenAlg}.
Examples are presented in Section~\ref{sec:moreExamples}.

\subsection{Intersection with a hypersurface}\label{Sec:RegenKey}
Given a witness set collection for $\cY$  and a hypersurface $\cG$, we compute a witness set collection for $\cY\cap\cG$. 
For the ease of exposition, we assume,
in this subsection, that no irreducible
component of $\cY$ is contained in $\cG$. 
The general case is provided in Section \ref{Sec:RegenAlg}.

This computation has two steps:
regenerate to union of hyperplanes
and then deform to the hypersurface $\cG$.
To simplify the presentation, we use
the following.

\begin{notation}
We denote $\delta_i$ to be the vector with $1$ in the $i^{th}$ position and $0$ elsewhere.
\end{notation}

\subsubsection{Regenerating to a union of hyperplanes}\label{Sec:RegenHyperplane}

Following the notation above, let $\cY\subset\ppp$
and $\deg \cG = (g_1,\dots,g_k)$.
For $i = 1,\dots,k$ and $j = 1,\dots,g_i$, 
let $s_i^{(j)}$ be a general linear polynomial
in the unknowns associated with~$\bP^{n_i}$.
Suppose that $\cS$ is the union of  $g_1+\cdots+g_k$ hyerplanes defined
by
\begin{equation}\label{defsij}
S = \prod_{i=1}^k \prod_{j=1}^{g_i} s_i^{(j)}.
\end{equation}
Let $\cS_{i}^{(j)}$ be the hyperplane defined
by $s_i^{(j)}$ so that
$$\cS = \bigcup_{i=1}^k \bigcup_{j=1}^{g_i} \cS_{i}^{(j)}.$$
With this setup, both $\cG$ and $\cS$ 
are hypersurfaces of degree $(g_1,\dots,g_k)$
and no irreducible component of $\cY$
is contained in either $\cG$ or $\cS$.
From $\mfW(\cY)$, the following computes
computes the isolated points in $\cY\cap\cS\cap\cL^{\bfd}$ for all possible $\bfd$, 
thereby producing $\mfW(\cY\cap\cS)$.

\begin{algorithm}\label{algCultivate}
[Regenerating to a union]
Given a witness set collection $\mfW(\cY)$ and $\left\{s_i^{(j)}\right\}$ whose product~$S$ defines the union of hyperplanes $\cS$, compute a witness set collection $\mfW(\cY\cap\cS)$.
\begin{enumerate}
\item Initialize $\mfW(\cY\cap\cS)$ to be the empty set. 
\item\label{itembfd} For $\bfd\in\mathbb{N}^k_{\geq 0}$ such that there exist $i\in\{1,2,\dots,k\}$ and $\bfW^\bfe(\cY)\in\mfW(\cY)$ 
with $\bfd+\delta_i=\bfe$:
\begin{enumerate}
\item Initialize $P^\bfd = \emptyset$ and append to $\mfW(\cY\cap\cS)$ the set
$$\bfW^\bfd(\cY\cap\cS):=\left\{Y\cup\{S\}, L^\bfd, P^\bfd\right\}.$$ 
\end{enumerate}
\item For $i = 1,\dots,k$ such that $g_i > 0$:
\begin{enumerate}
\item For each $\bfW^\bfe(\cY)\in\mfW(\cY)$ such
that $\bfe\geq\delta_i$:
\begin{enumerate}
\item\label{defM} Define $\bfd:=\bfe -\delta_i$ and hyperplane 
$\cM$ such that $\cL^\bfd \cap\cM= \cL^\bfe $.
\item For $j = 1,\dots,g_i$, append to $P^\bfd$ 
the endpoints of the homotopy 
$$
\bfHbar\left(\cV\cap\cL^\bfd,\cM\to \cS_{i}^{(j)}\right)=
\begin{cases}
V\\
L^{{\bfd}}\\
tM+\left(1-t\right)s_i^{(j)},
\end{cases}
$$
 with start points $\bfw^\bfe(\cY)$, which are the
isolated points of $\cY\cap\cL^\bfe$.
\end{enumerate}
\end{enumerate}
\item Return $\mfW(\cY\cap\cS)$.
\end{enumerate}
\end{algorithm}

\begin{remark}
In Item \ref{itembfd} of Algorithm \ref{algCultivate}, if $\cY$ is equidimensional then $|\bfd|=\dim\cV-1$ and $|\bfe|=\dim\cV$. 
\end{remark}

\begin{prop}\label{prop:Cultivation}
Algorithm~\ref{algCultivate} 
correctly
returns the multiprojective witness set collection 
$\mfW(\cY\cap\cS)$.
\end{prop}
\begin{proof}
By construction,
$$\cY\cap\cS\cap\cL^\bfd = \bigcup_{i=1}^k \bigcup_{j=1}^{g_i}
\cY\cap\cS_i^{(j)}\cap\cL^\bfd.$$
In Item \ref{defM}, the hyperplanes $\cM$ and $\cS_i^{(j)}$ are both 
defined by general linear polynomials
in the unknowns associated with $\bP^{n_i}$. 
Hence, the isolated points in
$\cY\cap\cS_{i}^{(j)}\cap\cL^\bfd$ 
are obtained by the homotopy 
$\bfHbar\left(\cY\cap\cL^\bfd,\cM\to \cS_{i}^{(j)}\right)$ starting at the
isolated points of $\cY\cap\cL^\bfe = \cY\cap\cM\cap\cL^\bfd$.
\end{proof}

\subsubsection{Deforming hypersurfaces}\label{Sec:RegenDeform}

The next step is to deform 
the hypersurface $\cS$ to $\cG$
along $\cY$ to compute $\mfW(\cY\cap\cG)$
from $\mfW(\cY\cap\cS)$.

\begin{algorithm}\label{algDeform}
[Deforming hypersurface]
Given $\mfW(\cY\cap\cS)$ and a hypersurface $\cG$
such that $\cS$ and $\cG$ have the same degree
and no irreducible component of $\cY$ is contained
in $\cG$, compute a witness set
collection~$\mfW(\cY\cap\cG)$.
\begin{enumerate}
\item Initialize $\mfW(\cY\cap\cG)$ to be the empty set.
\item For $\bfW^\bfd(\cY\cap\cS)\in\mfW(\cY\cap\cS)$ do
\begin{enumerate}
\item\label{homotopyQEnd} Set $Q^\bfd$ to the set of endpoints
of the homotopy 
$\bfHbar\left(\cY\cap\cL^\bfd,\cS\to\cG\right)$ starting at the isolated points of $\cY\cap\cS\cap\cL^\bfd$.
\item Remove from $Q^\bfd$ the nonisolated
points of $\cY\cap\cG\cap\cL^\bfd$, e.g., via \cite{LDT}.
\item Append to $\mfW(\cY\cap\cG)$ the multiprojective witness set 
$$
\bfW^\bfd(\cY\cap\cG):=\left\{
V\cup\{G\},
L^\bfd,
Q^\bfd
\right\}.$$ 
\end{enumerate}
\item Return $\mfW(\cY\cap\cG)$.
\end{enumerate}
\end{algorithm}

\begin{prop}\label{prop:Deform}
Algorithm~\ref{algDeform} 
correctly
returns the multiprojective witness set collection 
$\mfW(\cY\cap\cG)$.
\end{prop}
\begin{proof}
By \cite{CoeffParam}, 
the set of endpoints of 
$\bfHbar\left(\cY\cap\cL^\bfd,\cS\to\cG\right)$
starting at the isolated points of $\cY\cap\cS\cap\cL^\bfd$
consists of a superset of the isolated points
of $\cY\cap\cG\cap\cL^\bfd$.  Hence, 
removing the nonisolated points, e.g., using \cite{LDT},
yields the isolated points of $\cY\cap\cG\cap\cL^\bfd$.
\end{proof}

\begin{example}
Consider $\cV$ to be as in Ex.~\ref{parabola}
with $L^{(1,0)} = \{x_0 - 2x_1\}$
and $L^{(0,1)} = \{y_0 - y_1\}$. 
In particular, $\mfW(\cV) = \{\bfW^{(1,0)}(\cV),
\bfW^{(0,1)}(\cV)\}$
where $|\bfw^{(1,0)}(\cV)| = 1$
and $|\bfw^{(0,1)}(\cV)| = 2$.
Consider computing $\cV\cap\cG$
where $\cG$ is defined
by $27x_0y_0-50y_0x_1-25x_0y_1+50x_1y_1=0$
having degree $(g_1,g_2) = (1,1)$.

We can simplify Algorithm~\ref{algCultivate} by
taking $s_1^{(1)} = x_0 - 2x_1$
and $s_2^{(1)} = y_0 - y_1$.  Hence,
we have that
$\mfW(\cV\cap\cS) = \{\bfW^{(0,0)}(\cV\cap\cS)\}$
where $\bfw^{(0,0)}(\cV\cap\cS) = 
\bfw^{(1,0)}(\cV)\cup \bfw^{(0,1)}(\cV)$
consisting of three points.

For Algorithm~\ref{algDeform},
we only need to consider $\bfd = (0,0)$
for which Item~\ref{homotopyQEnd} 
deforms from the three paths 
defined by deforming $\cS$ to $\cG$ along
$\cV$ as shown
in Figure~\ref{fig:parabolaDegen}.
The set of three endpoints is $\cV\cap\cG$.

\begin{figure}
\floatbox[{\capbeside\thisfloatsetup{capbesideposition={left,top},capbesidewidth=8cm}}]{figure}[\FBwidth]
{\caption{\small{
Deforming from $\cS$ to~$\cG$ along the parabola
$\cV$ in the affine chart defined by $x_0=y_0=1.$}
}\label{fig:parabolaDegen}}
{\includegraphics[scale=0.35]{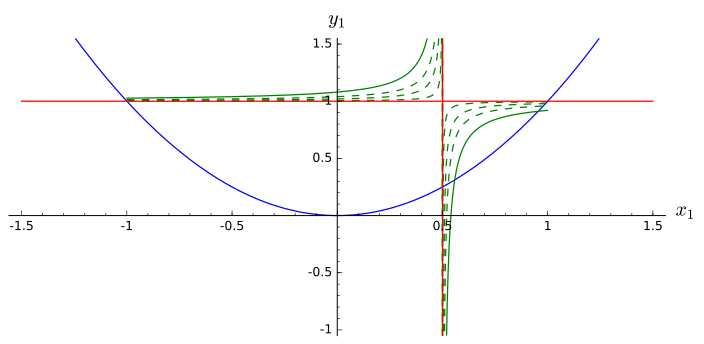}} 
\end{figure}
\end{example}

\subsection{Multiregeneration algorithm}\label{Sec:RegenAlg}

With Algorithms~\ref{algCultivate}
and~\ref{algDeform}, we are ready to describe
the full multiregeneration process for a polynomial
system $V := \{G_1,\dots,G_\ell\}$ 
which defines $\cV\subset\ppp$.  
The following computes witness set collections 
for $\cap_{i=1}^\iota \cG_i$ in sequence.

\begin{algorithm}\label{algRegeneration}
[Multiregeneration]

Given $V = \{G_1,\dots,G_\ell\}$, compute
a wintess set collection $\mfW(\cV)$.

\begin{enumerate}
\item Initialize a multiprojective witness set collection 
$\mfW(\ppp)$ by solving a linear system of equations $L^{(n_1,n_2,\dots,n_k)}=0$
\item For $\iota=0,\dots,\ell-1$
\begin{enumerate}
\item Define $\cX:=\cap_{i=1}^\iota\cG_i$ and $\cG:=\cG_{\iota+1}$.
\item Define $\cY$ to be the union of irreducible components of $\cX$ not contained in $\cG$. 
\item Define $\cZ$ to be the union of irreducible components of $\cX$ contained in $\cG$. 
\item Initialize the witness set collections  $\mfW(\cY)$ and $\mfW(\cZ)$ to each be the empty set. 
\item  For each $\bfW^\bfe(\cX)\in\mfW(\cX)$, do:
\begin{enumerate}
\item\label{defUbfe} Partition 
$\bfw^\bfe(\cX) = N^\bfe \sqcup U^\bfe$ where 
$N^\bfe = \bfw^\bfe(\cX)\setminus\cG$ and
$U^\bfe = \bfw^\bfe(\cX)\cap\cG$.
\item Append to $\mfW(\cY)$ the witness set 
$\bfW^\bfe(\cY):=\left\{X, L^\bfe, N^\bfe  \right\}$.
\item Append to $\mfW(\cZ)$ the witness set  
$\bfW^\bfe(\cZ):=\left\{X\cup\{G\}, L^\bfe, U^\bfe  \right\}$.
\end{enumerate}
\item For $\deg \cG = (g_1,\dots,g_k)$,
construct general linear forms  $\bigcup_{i=1}^k\left\{s_i^{(j)}:j=1,2,\dots,g_i\right\}$
and let~$S$ be their product as in $\eqref{defsij}$
which defines the hypersurface $\cS$.
\item Input $\mfW(\cY)$ and $\left\{s_i^{(j)}\right\}$
into Algorithm~\ref{algCultivate} to compute $\mfW(\cY\cap\cS)$.
\item Input $\mfW(\cY\cap\cS)$ and $\cG$ into  
 Algorithm~\ref{algDeform} to compute $\mfW(\cY\cap\cG)$.
\item Merge $\mfW(\cZ)$ and $\mfW(\cY\cap\cG)$ to
form $\mfW(\cap_{i=1}^{\iota+1} \cG_i)$.
\end{enumerate}
\item Return $\mfW(\cV) = \mfW(\cap_{i=1}^{\ell} \cG_i)$.
\end{enumerate}
\end{algorithm}

\begin{prop}\label{prop:RegenAlg}
Algorithm~\ref{algRegeneration} correctly
returns the multiprojective witness set
collection $\mfW(\cV)$.
\end{prop}
\begin{proof}
For each $\bfe$, 
we have $N^\bfe = \bfw^\bfe(\cY)$
so that Algorithms~\ref{algCultivate}
and~\ref{algDeform} compute
$\mfW(\cY\cap\cG)$.

The set $\cZ = \overline{
\cX\setminus{\cY}}$ is the union
of irreducible components of $\cX$ 
which are contained in $\cG$.  
For each~$\bfe$, $U^\bfe = \bfw^\bfe(\cZ)$.
Since each irreducible component of $\cZ$ is
an irreducible component of $\cX\cap\cG$,
we know that each point in $U^\bfe$ 
is an isolated point in $\cX\cap\cG\cap\cL^\bfe$.  
The result now follows immediately from
$$\cX\cap\cG\cap\cL^\bfe = 
\left({\cY}\cup\cZ\right)\cap\cG\cap\cL^\bfe
= \left(\cY\cap\cG\cap\cL^\bfe\right) \cup \left(\cZ\cap\cL^\bfe\right).$$

\end{proof}

\begin{remark}
The performance of Algorithm~\ref{algRegeneration}
is based on the ordering of the polynomials
and the structure of the polynomials. 
Example systems which can be viewed using different multihomogenizations are considered in Section~\ref{sec:moreExamples}.
\end{remark}

\begin{remark}
As formulated, some of the paths which need
to be tracked in Algorithm~\ref{algRegeneration}
may be singular.  
Following the notation in the proof of Prop.~\ref{prop:RegenAlg}, this occurs when
an irreducible component of $\cY$ 
for $\iota < \ell$ has multiplicity greater than one
with respect to $\{G_1,\dots,G_{\iota}\}$.
One can use deflation (see Remark~\ref{Deflation})
to produce paths with are nonsingular on $(0,1]$.  

Another option is to apply
Algorithm~\ref{algRegeneration}
to a randomization of $V$ (see Section~\ref{sec:Randomization}).
In this case, all paths which need to be tracked
are nonsingular on $(0,1]$ by Bertini's Theorem
so that one does not need to deflate paths.  
Two drawbacks of randomization, as discussed in 
Section~\ref{sec:Randomization},
are the typical destruction of 
structure and the increase of degrees
resulting in more paths which need 
to be tracked.
\end{remark}

The following is an illustrative 
example using Algorithm~\ref{algRegeneration}.

\begin{example}\label{exRegeneration}
Consider the following polynomial system defined on $\bP^2\times\bP^2$:
$$V :=\{G_1,G_2,G_3\}:= \left\{\begin{array}{ccc} x_0 y_2 - x_2 y_1, &
x_1 y_2 - x_2 y_1, &
x_0 y_1 y_2 - x_1 y_0 y_2
\end{array}\right\}.$$
It is easy to verify that $\cV$ has four irreducible components:
\begin{itemize}
\item $\cS_1 = \{([x_0,x_1,x_2],[1,0,0])\}$ having dimension $2$ and degree $1\omega^{(2,0)}$;
\item $\cS_2 = \{([x_0,x_1,0],[y_0,y_1,0])\}$ having dimension $2$ and degree $1\omega^{(1,1)}$;
\item $\cC_1 = \{([x_0,x_0,x_2],[y_0,y_0,y_2])~|~x_2y_0 = x_0y_2\}$ having dimension $1$ and degree $1\omega^{(1,0)}+1\omega^{(0,1)}$.
\item $\cC_2 = \{([0,0,1],[y_0,0,y_2])\}$ having dimension $1$ and degree $1\omega^{(0,1)}$;
\end{itemize}
We demonstrate using 
Algorithm~\ref{algRegeneration} to compute witness
sets for the pure-dimensional components, i.e., $\cS_1\cup\cS_2$
and $\cC_1\cup\cC_2$, with a decomposition computed
in Ex.~\ref{exRegenerationDecomp}.

\smallskip
\paragraph{$\iota=0$} Since $\cX = \bP^2\times\bP^2$ 
and $\cG=\cG_1$ a hypersurface of degree $(1,1)$, 
the witness point for $\bP^2\times\bP^2$
is regenerated into two points
as summarized in the following chart.
$$
\begin{array}{c|c|c}
\bfe & (2,1) & (1,2) \\
\hline
|\bfw^\bfe(\cG_1)| & 1 & 1
\end{array}
$$

\smallskip
\paragraph{$\iota=1$} For both $\bfd=(2,1)$ and $\bfd=(1,2)$, $U^\bfd = \emptyset$ so that all points are ``nonsolutions'' and must be regenerated.  
Since $\cG=\cG_2$ is a hypersurface of
degree $(1,1)$, two start points are regenerated into
four points as summarized
in the following chart.
$$
\begin{array}{c|c|c|c}
\bfe & (2,0) & (1,1) & (0,2) \\
\hline
|\bfw^\bfe(\cG_1\cap\cG_2)| & 1 & 2 & 1
\end{array}
$$

\smallskip
\paragraph{$\iota=2$}
For $\bfd = (2,0)$, $N^\bfd = \emptyset$ so that no points
are regenerated.  The point in $U^\bfd$ 
is the witness point for~$\cS_1$.  
For $\bfd = (1,1)$, both $U^\bfd$ and $N^\bfd$ consist
of one point.  The point in $U^\bfd$ is 
a witness point for~$\cS_2$.
Since $\cG_3$ is a hypersurface of degree $(1,2)$,
regenerating this point yields two nonisolated
endpoints, one each on $\cS_1$ and $\cS_2$,
and one isolated endpoint which is a witness point for~$\cC_1$.
For $\bfd = (0,2)$, $N^\bfd$ consists of one point 
which must be regenerated.  This yields
two isolated endpoints, one each on $\cC_1$ and $\cC_2$.
Since $\cV=\cG_1\cap\cG_2\cap\cG_3$, the results for $\cV$ are summarized 
in the following chart.
$$
\begin{array}{c|c|c|c|c}
\bfe & (2,0) & (1,1) & (1,0) & (0,1) \\
\hline
|\bfw^\bfe(\cV)| & 1 & 1 & 1 & 2
\end{array}
$$
\end{example}

\section{Multiregeneration Examples}\label{sec:moreExamples}

The following provides some examples demonstrating  multiregeneration described in Algorithm~\ref{algRegeneration}.

\subsection{Comparison on a zero-dimensional variety}

As mentioned in the introduction, 
multihomogeneous homotopies \cite{MultiHomHomotopy}
were developed by observing a bihomogeneous 
structure when solving the inverse kinematics
problem for 6R robots.  Following \cite[Ex.~5.2]{BHSW13}
with $\times$ denoting the cross product and~$\circ$~denoting 
the dot product, the polynomial system $f(z_2,z_3,z_4,z_5)$
defined on $(\bC^3)^4$ is 
{\footnotesize
$$
f_1:=z_1 \circ z_2 - c_1,\quad
f_2:=z_5 \circ z_6 - c_5 $$
$$
(f_3,f_4,f_5):=\left(a_1 z_1\times z_2 + d_2 z_2 + a_2 z_2\times z_3 + d_3 z_3 + a_3 z_3\times z_4 
+d_4 z_4+a_4z_4\times z_5 + d_5z_5 + a_5 z_5\times z_6 - p\right)
$$
$$
(f_6,f_7,f_8):=(z_2\circ z_3-c_2,z_3\circ z_4-c_3,z_4\circ z_5-c_4),\quad
(f_9,f_{10},f_{11},f_{12}):=(z_2\circ z_2 - 1, 
z_4\circ z_4 - 1,
z_3\circ z_3 - 1 ,
z_5\circ z_5 - 1 )
$$
}%
%
with variables $z_2,z_3,z_4,z_5\in\bC^3$ where
$z_1,z_6,p\in\bC^3$, $a_i,c_i,d_i\in\bC$ are general constants.  
In particular, $f$ consists of $12$ polynomials in $12$ variables and generically has $16$ roots.  
We have selected the ordering so that the first two polynomials
are linear, the next six are multilinear, and the last four are quadratic.

We consider solving $f=0$ using 
various multihomogeneous homotopies, 
a polyedral homotopy, and 
various multiregenerations.
In particular, we consider using a 1-, 2-, and 4-homogeneous
setup.  Each setup corresponds with homogenizing $f$
to yield a polynomial system $V$ defined on a 
product of 1, 2, and 4 projective spaces 
based on natural groupings of the variables,
namely $\{z_2,z_3,z_4,z_5\}$, 
$\{z_2,z_4\}\star\{z_3,z_5\}$, and $\{z_2\}\star\{z_3\}\star\{z_4\}\star\{z_5\}$, respectively.
The corresponding multihomogeneous B\'ezout counts 
are 1024, 320, and 576, while the BKK root count is 288.
These four counts are the total number of paths one needs 
to track using a 1-, 2-, and 4-homogeneous homotopy and a polyhedral homotopy, respectively.

In the multiregeneration, we only use slices which
could lead to isolated solutions yielding partial
information about the multidegrees in the intermediate
stages.  Tables~\ref{Table:Regen} and~\ref{Table:Regen4}
summarize the multiregeneration computations by listing
the total number of start points
and the total number of points in the resulting witness sets
for each codimension.  

\begin{table}[!t]
\centering
{\scriptsize
\begin{tabular}{c||c|c||c|c|c}
\multicolumn{1}{c||}{} & \multicolumn{2}{c||}{\mbox{1-homogeneous \cite{Regen}}} & \multicolumn{3}{c}{\mbox{2-homogeneous}}  \\
\cline{2-6}
codim & start points & witness points & start points & witness points & multidegree \\
\hline
\rule{0pt}{2.5ex}1 & 1 & 1 & 1 & 1 & $1\omega^{(5,6)}$\\
\hline
\rule{0pt}{2.5ex}2 & 1 & 1 & 1 & 1 & $1\omega^{(5,5)}$ \\
\hline
\rule{0pt}{2.5ex}3 & 2 & 2 & 2 & 2 & $1\omega^{(4,5)} + 1\omega^{(5,4)}$\\
\hline
\rule{0pt}{2.5ex}4 & 4 & 4 & 4 & 4 & 
$1\omega^{(3,5)} + 2\omega^{(4,4)} + 1\omega^{(5,3)}$
\\
\hline
\rule{0pt}{2.5ex}5 & 8 & 8 & 8 & 8 & 
$1\omega^{(2,5)} + 3\omega^{(3,4)} +
3\omega^{(4,3)} + 1\omega^{(5,2)}$\\
\hline
\rule{0pt}{2.5ex}6 & 16 & 16 & 14 & 14 & 
$4\omega^{(2,4)} + 6\omega^{(3,3)} + 4\omega^{(4,2)}$\\
\hline
\rule{0pt}{2.5ex}7 & 32 & 32 & 20 & 20 & 
$10\omega^{(2,3)} + 10\omega^{(3,2)}$\\
\hline
\rule{0pt}{2.5ex}8 & 64 & 62 & 20 & 20 & 
$20\omega^{(2,2)}$\\
\hline
\rule{0pt}{2.5ex}9 & 124 & 90 & 40 & 34 & 
$34\omega^{(1,2)}$\\
\hline
\rule{0pt}{2.5ex}10 & 180 & 62 & 68 & 28 & $28\omega^{(0,2)}$\\
\hline
\rule{0pt}{2.5ex}11 & 124 & 44 & 56 & 40 & $40\omega^{(0,1)}$ \\
\hline
\rule{0pt}{2.5ex}12 & 88 & {\bf 16} & 80 & {\bf 16} & $16\omega^{(0,0)}$ \\
\hline \hline
\rule{0pt}{2.5ex}total & 644 & & 314 & & 
\end{tabular}
}
\caption{Summary of multiregeneration solving 
the inverse kinematics of 6R robot using 
a $1$- and $2$-homogeneous setup. 
 }
\label{Table:Regen}
\end{table}

\begin{table}[!ht]
\centering
\begin{scriptsize}
\begin{tabular}{c||c|c|c}
\multicolumn{1}{c||}{} & \multicolumn{3}{c}{\mbox{4-homogeneous}} \\
\cline{2-4}
codim & start points & witness points & multidegree \\
\hline
\rule{0pt}{2.5ex}1 & 1 & 1 & 
$1\omega^{(2,3,3,3)}$\\
\hline
\rule{0pt}{2.5ex}2 & 1 & 1 & 
$1\omega^{(2,3,3,2)}$\\
\hline
\rule{0pt}{2.5ex}3 & 4 & 4 & 
$1\omega^{(1,3,3,2)}
+1\omega^{(2,2,3,2)}
+1\omega^{(2,3,2,2)}
+1\omega^{(2,3,3,1)}$\\
\hline
\rule{0pt}{2.5ex}\multirow{2}{*}{4} & \multirow{2}{*}{14} & \multirow{2}{*}{11} & 
$2\omega^{(1,2,3,2)}
+1\omega^{(1,3,2,2)}
+1\omega^{(1,3,3,1)}
+1\omega^{(2,1,3,2)}$\\
\rule{0pt}{2.5ex}& & & 
$+~2\omega^{(2,2,2,2)}
+1\omega^{(2,2,3,1)}
+1\omega^{(2,3,1,2)}
+2\omega^{(2,3,2,1)}$ \\
\hline
\rule{0pt}{2.5ex}\multirow{2}{*}{5} & \multirow{2}{*}{28} & \multirow{2}{*}{18} & 
$3\omega^{(1,2,2,2)}
+2\omega^{(1,2,3,1)}
+1\omega^{(1,3,1,2)}
+2\omega^{(1,3,2,1)}$\\
\rule{0pt}{2.5ex} & & & 
$+~3\omega^{(2,1,2,2)}
+1\omega^{(2,1,3,1)}
+3\omega^{(2,2,1,2)}
+3\omega^{(2,2,2,1)}$\\
\hline
\rule{0pt}{2.5ex}6 & 18 & 18 & 
$6\omega^{(1,1,2,2)}
+3\omega^{(1,1,3,1)}
+4\omega^{(1,2,1,2)}
+5\omega^{(1,2,2,1)}$\\
\hline
\rule{0pt}{2.5ex}7 & 18 & 16 &
$8\omega^{(1,1,1,2)}+
8\omega^{(1,1,2,1)}$\\
\hline
\rule{0pt}{2.5ex}8 & 16 & 14 & $14\omega^{(1,1,1,1)}$ \\
\hline
\rule{0pt}{2.5ex}9 & 28 & 24 & $24\omega^{(0,1,1,1)}$\\
\hline
\rule{0pt}{2.5ex}10 & 48 & 24 & $24\omega^{(0,0,1,1)}$\\
\hline
\rule{0pt}{2.5ex}11 & 48 & 20 & $20\omega^{(0,0,0,1)}$\\
\hline
\rule{0pt}{2.5ex}12 & 40 & {\bf 16} & $16\omega^{(0,0,0,0)}$\\
\hline \hline
\rule{0pt}{2.5ex}total & 264 & &
\end{tabular}
\end{scriptsize}
\caption{
Summary of multiregeneration solving 
the inverse kinematics of 6R robot using 
a $4$-homogeneous setup. 
}
\label{Table:Regen4}
\end{table}

\subsection{Comparison on a system with singular solutions}\label{Sec:Singular}
Our next example arises from computing the 
Lagrange points of a three-body system where the
small body is assumed to exert a negligible gravitational force 
on the two other bodies, e.g., Earth, Moon, and man-made satellite.
The nondimensionalized system $f$ defined
on $\bC^6$ derived in \cite[\S~5.5.1]{BHSW13},
with variables $\rho_1,w,\delta_{13},\delta_{23},x,y$
and parameter $\mu$, the ratio
of the masses of the two large bodies, is
{\small
$$
\left\{
\begin{array}{c}
w \rho_1 - 1,~~~
w \rho_2 - \mu, ~~~
(\rho_1 - x)^2 + y^2 - \delta_{13}^2,~~~
(\rho_2 + x)^2 + y^2 - \delta_{23}^2, \\
(w\delta_{13}^3\delta_{23}^3 - \mu \delta_{23}^3 - \delta_{13}^3)x +
\rho_1 \mu \delta_{23}^3 - \rho_2 \delta_{13}^3,~~~
(w\delta_{13}^3\delta_{23}^3 - \mu \delta_{23}^3 - \delta_{13}^3)y
\end{array}\right\}$$
}where $\rho_2 = 1 - \rho_1$.  
For generic $\mu$, this system has $64$ solutions counting multiplicity: 32~of multiplicity~1, 4 of multiplicity 2,
2 of multiplicity 3, and 2 of multiplicity 9.

We first consider solving $f=0$ using a
multiregeneration applied to $V$, 
a homogenization of $f$ based on using 
a $1$-homogeneous and a $5$-homogeneous structure
defined by $\{\rho_1\}\star\{w\}\star\{\delta_{13}\}\star\{\delta_{23}\}\star\{x,y\}$.
This 5-homogeneous structure was selected
since it minimizes the multihomogeneous B\'ezout count
amongst all possible partitions, namely 248
compared with the classical B\'ezout count of 1024.
Even though the system has singular solutions, 
they only arise at the final stage of the multiregeneration 
with this setup so that all paths which
need to be tracked are nonsingular~on~$(0,1]$.
The results are summarized
in Table~\ref{Table:RegenLagrange}
with the number of witness points counted with multiplicity.  

\begin{table}[!hb]
\centering
{\scriptsize
\begin{tabular}{c||c|c||c|c|c}
\multicolumn{1}{c||}{} & \multicolumn{2}{c||}{\mbox{1-homogeneous \cite{Regen}}} & \multicolumn{3}{c}{\mbox{5-homogeneous}} \\
\cline{2-6}
codim & start points & witness points & start points & witness points& multidegree \\
\hline
\rule{0pt}{2.5ex}1 & 2 & 2 & 2 & 2 & 
$1\omega^{(1,0,1,1,2)}+
1\omega^{(0,1,1,1,2)}$\\
\hline
\rule{0pt}{2.5ex}2 & 4 & 1 & 2 & 1 & $1\omega^{(0,0,1,1,2)}$\\
\hline
\rule{0pt}{2.5ex}3 & 2 & 2 & 4 & 4 & 
$2\omega^{(0,0,1,1,1)}
+2\omega^{(0,0,0,1,2)}$\\
\hline
\rule{0pt}{2.5ex}4 & 4 & 4 & 16 & 14 & 
$2\omega^{(0,0,1,1,0)}
+4\omega^{(0,0,1,0,1)}
+4\omega^{(0,0,0,1,1)}
+4\omega^{(0,0,0,0,2)}$\\
\hline
\rule{0pt}{2.5ex}5 & 32 & 28 & 48 & 48 & 
$10\omega^{(0,0,1,0,0)}
+10\omega^{(0,0,0,1,0)}
+28\omega^{(0,0,0,0,1)}$\\
\hline
\rule{0pt}{2.5ex}6 & 224 & {\bf 64} & 88 & {\bf 64} & $64
\omega^{(0,0,0,0,0)}$\\
\hline \hline
\rule{0pt}{2.5ex}total & 268 & & 160 & 
\end{tabular}
}
\caption{Summary of using multiregeneration
to solve the Lagrange points system}
\label{Table:RegenLagrange}
\end{table}

Due to the presence of singular solutions, we also
applied these two multiregenerations 
to the homogenization 
of a perturbation of $f$ as suggested 
in \cite{PerturbedRegen}, namely $f + \epsilon$ 
where $\epsilon\in\bC^6$ is general.  
From the isolated roots of $f + \epsilon$,
the isolated roots of $f$ are recovered using the 
parameter homotopy $f + t\cdot\epsilon = 0$.
The results matched those presented in 
Table~\ref{Table:RegenLagrange}
with the only change being the endpoints in 
the final stage are all nonsingular.

\subsection{Comparison on a positive-dimensional variety}

Our last example of multiregeneration
arises from computing a rank-deficiency set of 
a skew-symmetric matrix.  In particular, consider the system 
{\footnotesize
$$f(x,\lambda) = \left[\begin{array}{l}
S(x)\cdot B\cdot
\left[
\begin{array}{cccccc} 1 & 0 & \lambda_1 & \lambda_2 & \lambda_3 & \lambda_4 \end{array}\right]^T
 \\
S(x)\cdot B\cdot
\left[\begin{array}{cccccc} 0 & 1 & \lambda_5 & \lambda_6 & \lambda_7 & \lambda_8 \end{array}\right]^T
\end{array}\right]
\hbox{~~~where~~~}
S(x) = \left[\begin{array}{cccccc} 
0 & x_1 & x_2 & x_3 & x_4 & x_5 \\
-x_1 & 0 & x_6 & x_7 & x_8 & x_9 \\
-x_2 & -x_6 & 0 & x_{10} & x_{11} & x_{12} \\
-x_3 & -x_7 & -x_{10} & 0 & x_{13} & x_{14} \\
-x_4 & -x_8 & -x_{11} & -x_{13} & 0 & x_{15} \\
-x_5 & -x_9 & -x_{12} & -x_{14} & -x_{15} & 0 
\end{array}\right]
$$}and $B\in\bC^{6\times 6}$ is a fixed general matrix.

We focus on the skew-symmetric matrices $S(x)$ of rank $4$.
That is, we aim to compute 
$$\overline{\{(x^*,\lambda^*)\in\bC^{15}\times\bC^8~|~{\rm rank}~S(x^*) = 4 \hbox{~and~} \lambda^* \hbox{~is the unique solution to~}f(x^*,\lambda) = 0\}}.$$
Since~$f$ consists of $12$ polynomials and $\lambda\in\bC^8$, 
the codimension of the projection onto $x$ is at most~$4$.
Therefore, from a witness set computational standpoint, 
we can intrinsically restrict to a $4$-dimensional
linear space in $\bC^{15}$, namely by fixing
a general matrix $A\in\bC^{15\times4}$ and vector $b\in\bC^{15}$,
and considering the elements of $\bC^{15}$ of the form
$Ay + b$ where $y\in\bC^4$.  Thus, we consider the polynomial system
$$F(y,\lambda) = f(Ay+b,\lambda).$$
We consider polynomial systems $V$
arising from homogenizing $F$ using 
a $1$-homo\-geneous ($\bC^{12}$) 
and $2$-homogeneous ($\bC^4\times\bC^8$) structure.
Since each polynomial in $V$, respectively,
has degree $2$ and multidegree $(1,1)$, we 
apply multiregeneration to a randomization of $V$.
In particular, the 1-homogeneous regeneration setup
is  equivalent to the regenerative cascade \cite{HSW11}.
Table~\ref{Table:Regen2} summarizes the 
multiregenerations for each codimension 
by listing the total number of start points,
the total number of endpoints that satisfy
the system and are isolated ($U^\bfd$), 
the total number of nonisolated solutions (removed
in multiregeneration algorithm),
and the total number of ``nonsolutions'' ($N^\bfd$).
The only difference in this example between
the 1- and 2-homogeneous regenerations are the number
of start points of paths which need to be tracked
yielding two additional costs for the $1$-homogeneous
setup.  The first is the added cost in computing the 
additional start points themselves 
and the second is the typical added cost of tracking
these extra paths which often become ill-conditioned
near the end.

In this example, the solution set of $F = 0$ is actually 
an irreducible component of codimension $9$ and 
degree~$45$ in $\bC^{12}$ 
with multidegree 
$3\omega^{(3,0)}
+6\omega^{(2,1)}
+12\omega^{(1,2)}
+24\omega^{(0,3)}$
in $\bC^{4}\times\bC^{8}$.
The closure of the projection of this irreducible
component onto the $y$ coordinates, namely 
$$\overline{\{y\in\bC^{4}~|~{\rm rank~}S(Ay+b)\leq 4\}},$$
is a hypersurface of degree $3$
that is defined by the~cubic~polynomial~$\sqrt{\det S(Ay+b)}$.

\begin{table}
\centering
{\scriptsize
\begin{tabular}{c||c|c|c|c||c|c|c|c}
& \multicolumn{4}{c||}{\mbox{1-homogeneous \cite{HSW11}}} & \multicolumn{4}{c}{\mbox{2-homogeneous}}  \\
\cline{2-9}
codim & start points & iso. solns & noniso. solns & nonsolns & start points & iso. solns & noniso. solns & nonsolns \\
\hline
1 & 2 & 0 & 0 & 2 & 2 & 0 & 0 & 2 \\
\hline
2 & 4 & 0 & 0 & 4 & 4 & 0 & 0 & 4 \\
\hline
3 & 8 & 0 & 0 & 8 & 8 & 0 & 0 & 8 \\
\hline
4 & 16 & 0 & 0 & 16 & 16 & 0 & 0 & 16  \\
\hline
5 & 32 & 0 & 0 & 31 & 31 & 0 & 0 & 31 \\
\hline
6 & 62 & 0 & 0 & 57 & 57 & 0 & 0 & 57 \\
\hline
7 & 114 & 0 & 0 & 99 & 99 & 0 & 0 & 99 \\
\hline
8 & 198 & 0 & 0 & 163 & 163 & 0 & 0 & 163 \\
\hline
9 & 326 & 45 & 0 & 178 & 255 & 45 & 0 & 178 \\
\hline
10 & 356 & 0 & 84 & 104 & 260 & 0 & 84 & 104\\
\hline
11 & 208 & 0 & 66 & 24 & 136 & 0 & 66 & 24\\
\hline
12 & 48 & 0 & 20 & 0 & 24 & 0 & 20 & 0 \\
\hline \hline
total & 1374 & & & & 1055 & & & 
\end{tabular}
}
\caption{Summary of various regeneration procedures for positive-dimensional solving related to the rank-deficiency system}
\label{Table:Regen2}
\end{table}

\section{Decomposition and a trace test}\label{sectionTrace}
 
After computing witness point sets for multihomogeneous
varieties using multiregeneration described in Section~\ref{sec:multiregenerate}, 
the last remaining piece is to decompose into
multiprojective witness sets for each 
irreducible component.  One key part of this decomposition
is a {\em trace test}, first proposed
in \cite{TraceTest} for affine and projective varieties,
which validates that a collection of witness points 
forms a witness point set for a union of irreducible components.
We extend this to the multiprojective case in 
Section~\ref{sec:TraceTest}.  
The original motivation for such a test
was to verify the generic number of solutions to a parameterized
system with examples of this presented in Section~\ref{appTrace}.

Equipped with the membership test from Section~\ref{sec:Membership}, the trace test from Section~\ref{sec:TraceTest}, and monodromy~\cite{Monodromy},
we complete the decomposition in Section~\ref{sec:Decomposition}.


\subsection{Trace test}\label{sec:TraceTest}

Given a collection of witness points lying on a pure-dimensional
multiprojective variety~$V$, the goal is to verify that 
they form a collection of witness point sets for some 
variety $Z\subset V$.  We accomplish this by
generalizing the trace test proposed in \cite{TraceTest} for affine and projective varieties to the multiprojective 
setting.  To that end, for $V\subset\ppp$,
we first fix generic 
hyperplanes $\cH_i$ at infinity for each projective space 
$\mathbb{P}^{n_{i}}$.
Let $H_i$ be the polynomial defining $\cH_i$,
\begin{equation}\label{eq:H}
H = \prod_{i=1}^n H_i \hbox{~~~and~~~} \cH = \bigcup_{i=1}^n \cH_i
\end{equation}
so that $\cH$ is defined by the polynomial $H$.  

With this setup, we perform our trace using 
computations on the product space $\ppp$ 
with a {\em general coordinate}
from the Segre product space $\bP^{(n_1+1)\cdots(n_k+1)-1}$
defined as follows.

\begin{defn}
A {\em general coordinate} in $\bP^{(n_1+1)\cdots(n_k+1)-1}$
derived from $\ppp$ has the form
\begin{equation}\label{eq:GenCoordinate}
\rho:=(\Box x_0 + \cdots + \Box x_{n_1})\cdot(\Box y_0 + \cdots + \Box y_{n_2})\cdots(\Box z_0 + \cdots + \Box z_{n_k})
\end{equation}
where $(x,y,\dots,z)$ represents the coordinates
in $\ppp$.
\end{defn}

Hyperplanes in the Segre product space
$\bP^{(n_1+1)\cdots(n_k+1)-1}$
are defined by polynomials which are multilinear,
i.e., linear in the variables for each $\bP^{n_i}$.
We say that $\cR\subset\ppp$ is a {\em Segre linear slice}
defined by the polynomial $R$ if $R$ is multilinear.
For example, $\cH$ in \eqref{eq:H} is a Segre linear slice
since the polynomial $H$ in \eqref{eq:H} is multilinear.
With this, we are able to define the trace.

\begin{defn}\label{defnTrace}
Let $\cV\subset\ppp$ be a pure $c$-dimensional multiprojective
variety, $1\leq c' \leq c$, $\cR_1,\dots,\cR_{c'}$
be general Segre linear slices defined by $R_1,\dots,R_{c'}$,
respectively, and $\cL^\bfe$ be a linear space
of codimension $|\bfe|=c-c'$ defined by $L^\bfe$.  
Let ${\bf m}\subset \cV\cap\cL^\bfe\cap\cR_1\cap\cdots\cap\cR_{c'}$
and consider the homotopy
\begin{equation}\label{homotopyTrace}
{{\bf H}}^{{\bf e}}\left(\cV\cap\cL^\bfe,\cR\to \cR+\cH\right)=\begin{cases}
V\\
L^{{\bf e}}\\
R_{1}+\left(1-t\right)H\\
\vdots \\
R_{c'}+\left(1-t\right)H
\end{cases}
\end{equation}
where $H$ as in \eqref{eq:H}.  
For each $m\in{\bf m}$, let $m(t)$ denote the path 
defined by this homotopy starting at $m$.  
Then, the \textit{trace} of ${\bf m}$
with respect to ${{\bf H}}^{{\bf e}}$ and a general coordinate $\rho$ is the average of the $\rho$ coordinate of the paths $m(t)$, i.e.,
$$\trace_{\bf m}^\rho{{\bf H}}^{{\bf e}}\left(\cV\cap\cL^\bfe,\cR\to \cR+\cH\right)(t) :=\frac{1}{|{\bfm}|}\sum_{m\in{\bf m}}\rho(m(t)).$$
The trace is said to be \textit{affine linear} with respect to $\cH_1,\dots,\cH_k$ if it is a linear function in $t$
when restricted to the affine chart where $H_1=\cdots=H_k=1$. 
\end{defn}

We illustrate with the following example.

\begin{example}\label{ex:Trace}
Consider the pure $1$-dimensional variety $\cV\subset\bP^2\times\bP^1$ defined by
$$V=\{(y_1^2-y_0^2) x_1 - y_1^2 x_0 ,\,\,\, (y_1^2-y_0^2) x_2 - y_1^2 x_0\}.$$ 
In particular, $\cV$ is the union of three irreducible varieties: 
\begin{itemize}
\item $\cL_1 = \{([0,x_1,x_2],[1,1])\}$,
\item $\cL_2 = \{([0,x_1,x_2],[1,-1])\}$, and
\item $\cC = \{([x_0,x_1,x_1],[y_0,y_1])~|~(y_1^2-y_0^2)x_1 = y_1^2 x_0\}$.
\end{itemize}
For simplicity, we take $H_1 = x_2$ and $H_2 = y_1$,
and the Segre linear slice $\cR_1$ be defined by
$$R_1= (6x_0/7 + 3x_1/5+ 2x_2/7)(y_0-y_1/2).$$
The set $\cV\cap\cR_1$ consists of five points $m_1,\dots,m_5$,
say where $\cL_1\cap\cR_1 = \{m_1\}$,
$\cL_2\cap\cR_1 = \{m_2\}$,
and $\cC\cap\cR_1 = \{m_3,m_4,m_5\}$.
With the general coordinate
$$\rho = (2x_0/7 - 5x_1/12 + 3x_2/17)(4y_0/13 - 3y_1/14),$$
we consider the homotopy 
$${{\bf H}}^{{\bf e}}\left({\cV},\cR_1\to \cR_1+H\right)
\begin{cases}
(y_1^2-y_0^2)x_1 - y_1^2 x_0 \\
(y_1^2-y_0^2)x_2 - y_1^2 x_0 \\
(6x_0/7+3x_1/5+2x_2/7)(y_0-y_1/2) + (1-t)x_2y_1.
\end{cases}
$$
Restricting to $x_2 = y_1 = 1$,
Figure~\ref{fig:SegreTrace} plots 
the trace of this homotopy for two sets of
start points, namely $\{m_1,m_2,m_3,m_4\}$
and $\{m_1,m_2,m_3,m_4,m_5\}$.
This plot shows that the trace of $\{m_1,m_2,m_3,m_4\}$
is not affine linear while the trace 
$\{m_1,m_2,m_3,m_4,m_5\}$ is indeed affine linear.

\begin{figure}[!htb]
\centering
\includegraphics[scale=0.3]{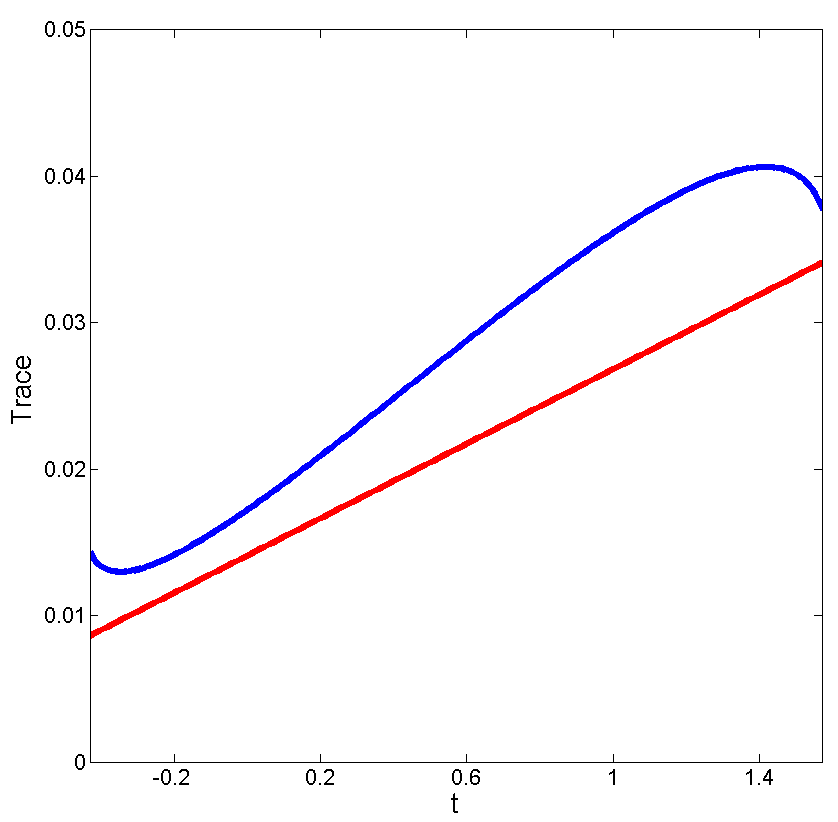}
\caption{Plot of the trace for $\{m_1,\dots,m_4\}$ 
in blue (nonlinear) and $\{m_1,\dots,m_5\}$ in red (linear).}
\label{fig:SegreTrace}
\end{figure}
\end{example}

As with the classical trace test \cite{TraceTest},
the trace with start points ${\bf m}$ is affine linear if and only if ${\bf m}$ is a collection of witness points
for a multiprojective variety.

\begin{theorem}[Trace test]\label{thm:TraceTest}
Suppose that $\cV\subset\ppp$ is a pure $c$-dimensional
variety defined by $V$, 
$\cH_i$ is a general hyperplane
at infinity defined by $H_i$ for $\bP^{n_i}$, 
$\rho$ is a general coordinate,
$\cR_1,\dots,\cR_{c'}$ are general Segre linear slices 
defined by $R_1,\dots,R_{c'}$ for $1\leq c'\leq c$,
$\cL^\bfe$ is a linear space of codimension $|\bfe| = c-c'$
defined by $L^\bfe$, and ${\bf m}\subset\cV\cap\cL^\bfe\cap\cR_1\cap\cdots\cap\cR_{c'}$.
Then, $\trace_{\bf m}^\rho{{\bf H}}^{{\bf e}}\left(\cV\cap\cL^\bfe,\cR\to \cR+\cH\right)(t)$ is affine linear in $t$
if and only if there exists $Z\subset\cV\cap\cL^\bfe$,
which is a union of irreducible components
of dimension $c'$, such that
$Z\cap\cR_1\cap\cdots\cap\cR_{c'} = {\bf m}$.
\end{theorem}
\begin{proof}
By applying \cite[Thm.~3.6]{TraceTest}
to the Segre product space $\bP^{(n_1+1)\cdots(n_k+1)-1}$,
one has the result if the trace in every coordinate 
of $\bP^{(n_1+1)\cdots(n_k+1)-1}$ is affine linear.  
Clearly, if the trace is affine linear in every
coordinate of $\bP^{(n_1+1)\cdots(n_k+1)-1}$,
then it is affine linear in the general coordinate $\rho$.

Conversely, we know, for all $\alpha_{a_1},\beta_{a_2},\dots,\gamma_{a_k}\in\bC$,
$$\sum_{(i_1,\dots,i_k)} (\alpha_{i_1}\cdot\beta_{i_2}\cdots \gamma_{i_k}) q_{i_1,\dots,i_k} = 0$$
if and only if every $q_{i_1,\dots,i_k} = 0$.  
Hence, if there exists \mbox{$q_{j_1,\dots,j_k} \neq 0$}, then
$$\sum_{(i_1,\dots,i_k)}\left(\alpha_{i_1}\cdot \beta_{i_2} \cdots \gamma_{i_k}\right) q_{i_1,\dots,i_k} \neq 0$$ 
for general $\alpha_{a_1},\beta_{a_2},\dots,\gamma_{a_k}\in\bC$.
Therefore, if the trace of some coordinate in the
Segre product space $\bP^{(n_1+1)\cdots(n_k+1)-1}$
is not affine linear in $t$,
then the trace for $\rho$ will also not be affine
linear in $t$.
\end{proof}

In order to numerically test for affine linearity in $t$,
a classical approach is to evaluate
the trace at~$3$ distinct values of $t$ and decide
if they lie on a line.  Another approach
is described in \cite{TraceDerivs} which computes
derivatives with respect to $t$ of the trace.
Two savings in this numerical computation 
are to work intrinsically on the hyperplanes 
$H_i = 1$ to reduce the number of variables, and 
to use slices which preserve structure and/or simplify
the computation.  These savings are 
utilized in Section \ref{appTrace}.

\subsection{Decomposition}\label{sec:Decomposition}

Decomposition of a pure-dimensional variety
corresponds with partitioning the collection
of witness point sets into collections
corresponding to each irreducible component.
Similar to the classical trace test \cite{TraceTest},
one can use Theorem~\ref{thm:TraceTest} to
perform this partition by finding the smallest subsets for 
which the trace is affine linear.  
When the total number of points is small,
such as in Ex.~\ref{ex:Trace} which has only
$5$ points to consider, this is an effective approach.
However, this quickly becomes impractical
as the number of points increases.
Thus, we propose using two methods to 
reduce the number of possible partitions one needs
to consider.  Both are based on the fact that
smooth points of irreducible components are 
path connected.

The first approach is currently used in the
classical affine or projective setting, namely
to utilize random monodromy loops \cite{Monodromy}.
That is, one tracks the points along the variety 
as a linear slice is moved in a general loop in the 
corresponding Grassmannian.  Each start point
and corresponding endpoint defined by this loop
must lie on the same irreducible component.

In the multiprojective setting, 
there is a second approach that one may utilize,
namely to use one set of slices 
to perform membership testing (Section~\ref{sec:Membership}) 
on the points arising from a different set of slices.
As with monodromy loops, all points which are connected
by smooth paths must lie on the same irreducible component
thereby reducing the possible number of partitions to
consider.

\begin{example}\label{exRegenerationDecomp}
The multiregeneration computation in Ex.~\ref{exRegeneration}
yielded pure-dimensional witness sets
for dimensions $1$ and $2$.  We now illustrate
how to decompose into the four irreducible components.

We first start with the pure $2$-dimensional variety
which has degree $1\omega^{(2,0)}+1\omega^{(1,1)}$.
Since the trace of $m_{(2,0)}$, 
which is the unique point in $\cV\cap\cL^{(2,0)}$,
is affine linear, this shows that there are two irreducible
components, one of degree $1\omega^{(2,0)}$ (namely, $\cS_1$)
and one of degree $1\omega^{(1,1)}$ (namely, $\cS_2$).

We now turn to the pure $1$-dimensional variety
which has degree $1\omega^{(1,0)}+2\omega^{(0,1)}$.
For simplicity and 
concreteness, let $H_1 = x_2$ and $H_2 = y_2$
define the hyperplanes as infinity $\cH_1$ and $\cH_2$, respectively.  Let $L^{(1,0)} = x_0 + x_1 - 2x_2$
and $L^{(0,1)} = y_0 - 2y_1 - y_2$ define
the linear slices $\cL^{(1,0)}$ and $\cL^{(0,1)}$, respectively,
with $R_1 = L^{(1,0)}\cdot L^{(0,1)}$ define
the Segre linear slice $\cR_1$.  Thus, $\cV\cap\cR_1$ 
consists of $3$ isolated points
$$m_1 = ([1,1,1],[1,1,1]), ~~ m_2 = ([-1,-1,1],[-1,-1,1]), \hbox{~~and~~} m_3 = ([0,0,1],[1,0,1])$$
where $\{m_1\}$ and $\{m_2,m_3\}$ are the sets
of isolated points of $\cV\cap\cL^{(1,0)}$ and $\cV\cap\cL^{(0,1)}$, respectively.
We now employ the membership test (see Section~\ref{sec:Membership}) applied to $m_1$ 
using $\cL^{(0,1)}$.  For simplicity,
we take $L^{(0,1)}_{m_1} = (1+\sqrt{-1})(y_0 - 2y_1 + y_2)$
which defines $\cL^{(0,1)}_{m_1}$ and utilize
the homotopy $\bfHbar(\cV,\cL^{(0,1)}\to \cL^{(0,1)}_{m_1})$
with start points $m_2$ and $m_3$.  
The endpoints of the path starting at $m_2$ and $m_3$
are $m_1$ and $([0,0,1],[-1,0,1])$, respectively.
Hence, $m_1$ and $m_2$ lie on the same irreducible component.  
The trace test shows that there are indeed
two irreducible components, one from $\{m_1,m_2\}$ 
having degree $1\omega^{(1,0)}+1\omega^{(0,1)}$ (namely $\cC_1)$
and the other from $\{m_3\}$ having degree 
$1\omega^{(0,1)}$ (namely $\cC_2)$.
\end{example}

\section{Trace test examples}\label{appTrace}

We close with several examples using 
the multihomogeneous trace test 
from Section~\ref{sectionTrace}.

\subsection{Tensor decomposition}\label{sec:Tensor}

In \cite{CMO14}, the problem of computing
the number $k$ of tensor decompositions of a general tensor
of rank $8$ in $\bC^3\otimes\bC^6\otimes\bC^6$ is formulated
with \cite[Thm.~3.5]{CMO14} proving $k\geq6$.
Computation~4.2 of~\cite{Tensors} uses numerical algebraic geometry
together with the fact the number
of minimal decompositions of
a general tensor of $\Sym^3\bC^3\otimes\bC^2\otimes\bC^2$ 
is also $k$ to conclude that $k = 6$.  
We use the trace test from Theorem~\ref{thm:TraceTest}
to confirm this result.  To that end,
we first consider a natural formulation as an affine variety.
For $a,b\in\bC$, Let $\nu_3(1,a,b)\in\bC^{10}$ be the degree $3$ Veronese embedding, i.e.,
$$v_3(1,a,b) = (1,a,b,a^2,ab,b^2,a^3,a^2b,ab^2,b^3).$$
For $x\in\bC^5$, consider the tensor $T(X) = \nu_3(1,x_1,x_2)\otimes(x_3,x_4)\otimes(1,x_5)\in\bC^{40}$.  
For general
 $P,Q\in\bC^{40}$, we consider the set
$$\cC = \left\{\left(s,X^{(1)},\dots,X^{(8)}\right)~\left|~P\cdot s + Q = \sum_{i=1}^8 T\left(X^{(i)}\right)\right\}\right.\subset\bC\times\left(\bC^5\right)^8.$$
It follows from \cite[Thm.~3.42]{SS85} that $\cC$ is 
an irreducible variety of dimension $1$.
Using the natural embedding of $\bC\times(\bC^5)^8\hookrightarrow\bP^1\times\bP^{40}$,
we have the two natural hyperplanes at infinity and
can restrict each being set to $1$ resulting
in computations back on $\cC\subset\bC\times(\bC^5)^8$.
In order to maintain the 
invariance under the symmetric group on $8$ elements, namely $S_8$, for $r_1,r_2\in\bC$,
we consider the Segre linear space~$\cR_1$ 
defined by 
$$R_1 = (s - r_1)\left(x^{(1)}_5 + \cdots +x^{(8)}_5 - r_2\right).$$
After randomly selecting $P,Q\in(\bC^{5})^8$ and $r_1,r_2\in\bC$,
we used {\tt Bertini} \cite{Bertini} to compute $\cC\cap\cR_1$,
which yielded a total of 
$528\cdot 8! = \hbox{21,288,960}$
points with precisely $6\cdot 8!$ points 
satisfying $s - r_1=0$, i.e.,~$k = 6$.
To verify that we have computed every point
in $\cC\cap\cR_1$, we applied the trace test from
Theorem~\ref{thm:TraceTest} which showed
that the trace was indeed affine linear
confirming \cite[Computation~4.2]{Tensors}.

\subsection{Alt's problem}\label{sec:Alt}

In 1923, Alt \cite{Alt} formulated the problem of counting the number of coupler curves generated by four-bar linkages
that pass through nine general points in the plane.
The solution to this problem, namely $1442$, 
was found using homotopy continuation in \cite{NinePoint}.
We use the trace test from Theorem~\ref{thm:TraceTest}
to confirm this result and study two related systems.

Following the formulation in \cite{NinePoint}, only the displacements
between one selected point and the other eight points are of interest. 
These displacements are represented using isotropic coordinates
$(\delta_j,\overline{\delta_j})\in\bC^2$ for $j = 1,\dots,8$.  
The polynomial system under consideration consists of $12$ polynomials
in $12$ variables: 
$$a,b,m,n,x,y,\overline{a},\overline{b},\overline{m},\overline{n},\overline{x},\overline{y}.$$
The first four polynomials are:
$$f_1 = n - a \overline{x}, ~~~~ f_2 = \overline{n} - \overline{a} x, ~~~~
f_3 = m - b \overline{y}, ~~~~ f_4 = \overline{m} - \overline{b} y.$$
Next, for $j = 1,\dots,8$, we have $f_{4+j} = \gamma_j \overline{\gamma_j} + \gamma_j \gamma_j^0 + \overline{\gamma_j} \gamma_j^0$ where:
$$
\begin{array}{llllllllllll}
\gamma_j&:=& q_j^x r_j^y - q_j^y r_j^x, &
\overline{\gamma_j}&:=&r_j^x p_j^y - r_j^y p_j^x, &
\gamma_j^0&:=& p_j^x q_j^y - p_j^y q_j^x,\quad\text{and}\\
p_j^x &:=& \overline{n} - \overline{\delta_j} x, & 
q_j^x &:=& n - \delta_j \overline{x}, & 
r_j^x &:=& \delta_j (\overline{a} - \overline{x})
+ \overline{\delta_j} (a - x) - \delta_j \overline{\delta_j}, \\
p_j^y &:=& \overline{m} - \overline{\delta_j} y, & 
q_j^y &:=& m - \delta_j \overline{y}, & 
r_j^y &:=& \delta_j (\overline{b} - \overline{y})
+ \overline{\delta_j} (b - y) - \delta_j \overline{\delta_j}.
\end{array}
$$

\subsubsection{Original problem}
To study Alt's problem, we randomly selected $(\delta_j,\overline{\delta_j})\in\bC^2$
for $j = 1,\dots,7$ and $u_\ell,v_\ell\in\bC$ for $\ell = 1,2$.
For $\delta_8 = u_1 s + v_1$ and $\overline{\delta_8} = u_2 s + v_2$, 
we consider the irreducible variety 
$\cC\subset\bC\times\bC^{12}$ 
defined by $f_1 = \cdots = f_{12} = 0$
consisting of nondegenerate four-bar linkages whose
coupler curve passes through $(\delta_j,\overline{\delta_j})$ for $j = 1,\dots,8$.
As in Section~\ref{sec:Tensor}, 
we actually consider the natural embedding
of $\bC\times\bC^{12}\hookrightarrow\bP^1\times\bP^{12}$
and restrict each natural hyperplane at infinity 
to be $1$.  To further simplify the computation,
we actually model the computation 
related to computing the closure of
the image of $\cC$ under the map
$$\pi(s,a,b,m,n,x,y,\overline{a},\overline{b},\overline{m},\overline{n},\overline{x},\overline{y}) = (s,x).$$
That is, for random $r_1,r_2\in\bC$, we consider 
intersecting $\cC$ with the Segre linear space $\cR_1$
defined by 
$$R_1 = (s - r_1)(x - r_2).$$
We used {\tt Bertini} to compute $\cC\cap\cR_1$
resulting in 32,358 points with
precisely $8652 = 1442\cdot2\cdot3$ of them
satisfying $s - r_1=0$.  In particular, $8652$
corresponds to the number of distinct four-bar
mechanisms while $1442$ is the number of distinct
coupler curves.  The $6$-fold reduction
from mechanisms to coupler curves
is due to a two-fold symmetry and Roberts cognates.
Since the trace of the 32,358 points is indeed
affine linear, this confirms
the result of \cite{NinePoint}.

\subsubsection{Product decomposition bound}
A product decomposition bound is constructed in \cite{ProductDecomp} by replacing
$$f_{4+j} = \gamma_j \overline{\gamma_j} + \gamma_j \gamma_j^0 + \overline{\gamma_j} \gamma_j^0$$
for $j = 1,\dots,8$ with
$$g_{4+j} = (\alpha_j \gamma_j + \beta_j \gamma_j^0)\cdot(\mu_j \overline{\gamma_j} + \nu_j \gamma_j^0)$$
where $\alpha_j,\beta_j,\mu_j,\nu_j\in\bC$.
For generic choices, \cite{ProductDecomp}
showed that the resulting system has 
18,700 isolated solutions, which is 
a product decomposition 
bound on the number of isolated solutions
for Alt's problem.  We can repeat
a similar computation for Alt's original problem
with this product decomposition system
by letting $\cC$ denote the
the union of the irreducible components of dimension $1$
defined by randomly selecting all parameters 
except $\delta_8$.  We then intersected $\cC$
with the Segre linear slice $\cR_1$ defined~by
$$R_1 = (\delta_8 - r_1)(x-r_2)$$
for random $r_1,r_2\in\bC$.  
Using {\tt Bertini}, we find that 
$\cC\cap\cR_1$ consists of 37,177 points,
precisely 18,700 satisfy $\delta_8 - r_1 = 0$.
We confirm the result of~\cite{ProductDecomp}
since the trace of the 37,177 points is
affine linear.

\subsubsection{Polyhedral bound}
In our last example, we consider
the family of systems $\sF$ with the same monomial support 
as the polynomial system $f_1,\dots,f_{12}$
defining Alt's problem as above.  The number of solutions
to a generic member of $\sF$ is called the
polyhedral bound or BKK bound \cite{NinePtBKK}.
We consider a line in $\sF$ by
randomly fixing all coefficients except the coefficient
of $\overline{a}x$ in $f_2$, which we call $p_2$.  
Consider the set of isolated solutions 
over this line yields a variety of dimension $1$
which we intersect with $\cR_1$ defined by
$$R_1 = (p_2-r_1)(x-r_2)$$
for random $r_1,r_2\in\bC$.  
Using {\tt Bertini},
we obtain 132,091 points, which was verified to be complete
by the trace test, 
with exactly 79,135 satisfying $p_2-r_1 = 0$.  
Therefore, our computation shows that the BKK (polyhedral)
bound for this system is 79,135.  Since this contradicts
the bound of 83,977 reported in \cite{NinePtBKK}, 
we provide a proof based on exact computations
in {\tt polymake} \cite{polymake}
that 79,135 is indeed correct.

\begin{theorem}[BKK bound for Alt's problem]\label{thm:NinePtBKK}
The BKK (polyhedral) bound for $f_1,\dots,f_{12}$
as above for Alt's problem is equal to 79,135.
\end{theorem}
\begin{proof}
Following \cite{NinePtBKK}, we can use 
$f_1,\dots,f_4$ to remove $n$, $\overline{n}$, $m$, and
$\overline{m}$ resulting in an unmixed system,
i.e., all polynomials have the same monomial support, 
consisting of $8$ polynomials of degree $7$ in $8$ variables
with the same BKK (polyhedral) bound as the original system.
Since $\gamma_j$, $\overline{\gamma_j}$,
and $\gamma_j^0$ are quartic polynomials, 
one naively would have expected the polynomials to have degree $8$, which is not the case due to exact cancellation 
in the coefficients.  Hence, to properly recover the 
monomial support, we used the following computation
in {\tt Maple} to find that the support of these $8$ polynomials 
is $239$ monomials.  

\begin{verbatim}
  #input
  n := a*xHat: nHat := aHat*x: m := b*yHat: mHat := bHat*y:
  px := nHat - x*dHat: py := mHat - y*dHat:
  qx := n - xHat*d: qy := m - yHat*d:
  rx := d*(aHat - xHat) + dHat*(a - x) - d*dHat:
  ry := d*(bHat - yHat) + dHat*(b - y) - d*dHat:
  g := qx*ry - qy*rx:
  gHat := rx*py - ry*px:
  gZero := px*qy - py*qx:
  f := g*gHat + g*gZero + gHat*gZero:
  nops([coeffs(expand(f),[a,b,x,y,aHat,bHat,xHat,yHat])]);
  239    #output
\end{verbatim}

We then used the software {\tt polymake}
\cite{polymake} to compute the vertices
of the polytope which is the 
convex hull of these $239$ monomials
resulting in $150$ vertices.  
We note that \cite{NinePtBKK} reported $259$ monomials
and $158$ vertices with the lower values
in our computation possibly resulting from 
using symbolic computations in {\tt Maple} to 
remove monomials whose coefficients are identically 
zero.  
To complete the proof, we used
{\tt polymake} to compute the volume of this
polytope, which was $2261/1152$.  Hence, the BKK (polyhedral)
bound is equal to 
$$8!\cdot 2261/1152 = \hbox{79,135}.$$
\end{proof}

\section*{Acknowledgment}

The authors want to thank Anton Leykin, Andrew Sommese, and Frank
Sottile for helpful conversations, some of which occurred
during the fall of 2014 at the Simons Institute for the Theory of Computing, leading to the results in this paper. 
Certain aspects of the trace test first presented in earlier
drafts of this paper have subsequently been reformulated in \cite{LS16}.

JDH was supported in part by NSF ACI 1460032 and Sloan Research 
Fellowship.  
JIR was supported in part by NSF DMS 1402545.

\bibliographystyle{amsplain}
\providecommand{\bysame}{\leavevmode\hbox to3em{\hrulefill}\thinspace}
\providecommand{\MR}{\relax\ifhmode\unskip\space\fi MR }
\providecommand{\MRhref}[2]{%
  \href{http://www.ams.org/mathscinet-getitem?mr=#1}{#2}
}
\providecommand{\href}[2]{#2}


\end{document}